\documentclass[a4paper,12pt]{article} 
\title{Logarithmic intersections of double ramification cycles}

\usepackage{amsmath, amssymb, mathrsfs, amsthm, shorttoc, geometry, stmaryrd,wrapfig, epigraph}
\usepackage{mathtools} 
\usepackage{hyperref}
\usepackage{xcolor}
\hypersetup{
colorlinks,
linkcolor={black},
citecolor={blue!50!black},
urlcolor={blue!80!black}
}

\usepackage[all]{xy}

\usepackage{tabularx}
\usepackage{longtable}
\numberwithin{equation}{subsection}

\newcommand{\mat}[1]{\begin{bmatrix}#1\end{bmatrix}}

\usepackage{enumitem}

\usepackage{cleveref}
\crefformat{equation}{(#2#1#3)}

\AtBeginDocument{\renewcommand{\ref}[1]{\cref{#1}}}

\usepackage{pgf,tikz}
\usetikzlibrary{matrix, calc, arrows}

\usepackage{tikz-cd} 

\makeatletter
\newcommand*{\doublerightarrow}[2]{\mathrel{
  \settowidth{\@tempdima}{$\scriptstyle#1$}
  \settowidth{\@tempdimb}{$\scriptstyle#2$}
  \ifdim\@tempdimb>\@tempdima \@tempdima=\@tempdimb\fi
  \mathop{\vcenter{
    \offinterlineskip\ialign{\hbox to\dimexpr\@tempdima+1em{##}\cr
    \rightarrowfill\cr\noalign{\kern.5ex}
    \rightarrowfill\cr}}}\limits^{\!#1}_{\!#2}}}
\newcommand*{\triplerightarrow}[1]{\mathrel{
  \settowidth{\@tempdima}{$\scriptstyle#1$}
  \mathop{\vcenter{
    \offinterlineskip\ialign{\hbox to\dimexpr\@tempdima+1em{##}\cr
    \rightarrowfill\cr\noalign{\kern.5ex}
    \rightarrowfill\cr\noalign{\kern.5ex}
    \rightarrowfill\cr}}}\limits^{\!#1}}}
\makeatother

\newcommand{\on}[1]{\operatorname{#1}}
\newcommand{\bb}[1]{{\mathbb{#1}}}

\newcommand{\ca}[1]{{\mathcal{#1}}}
\newcommand{\bd}[1]{{\mathbf{#1}}}

\newcommand{\ul}[1]{{\underline{#1}}}

\newcommand{\Span}[1]{\left<#1\right>}

\newcommand{\hra}{\hookrightarrow}
\newcommand{\sub}{\subseteq}
\newcommand{\tra}{\rightarrowtail}

\theoremstyle{definition}
\newtheorem{definition}{Definition}[section]

\theoremstyle{plain}

\newtheorem{lemma}[definition]{Lemma}
\newtheorem{theorem}[definition]{Theorem}
\newtheorem{corollary}[definition]{Corollary}

\theoremstyle{remark}

\newtheorem{remark}[definition]{Remark}

\newtheorem{example}[definition]{Example}

\usepackage{letltxmacro}
\LetLtxMacro{\phiorig}{\phi}
\renewcommand{\phi}{\varphi}

\newcommand{\loz}{\lozenge}
\newcommand{\bloz}{\blacklozenge}

\newcommand{\Dcomment}[1]{{\color{blue}D: #1}}

\author{David Holmes and Rosa Schwarz}
\date{\today}

\newcounter{nootje}
\newcommand{\beq}{\begin{equation}}
\newcommand{\eeq}{\end{equation}}
\newcommand{\beqs}{\begin{equation*}}
\newcommand{\eeqs}{\end{equation*}}

\renewcommand{\k}{k}

\tikzset{
  symbol/.style={
    draw=none,
    every to/.append style={
      edge node={node [sloped, allow upside down, auto=false]{$#1$}}}
  }
}

\begin{document}
\maketitle


\begin{abstract} 
We describe a theory of logarithmic Chow rings and tautological subrings for logarithmically smooth algebraic stacks, via a generalisation of the notion of piecewise-polynomial functions. Using this machinery we prove that the double-double ramification cycle lies in the tautological subring of the (classical) Chow ring of the moduli space of curves, and that the logarithmic double ramification cycle is divisorial (as conjectured by Molcho, Pandharipande, and Schmitt). 
\end{abstract}


\tableofcontents

\newcommand{\Mtildes}{ \widetilde{\ca M}^\Sigma}
\newcommand{\sch}[1]{\textcolor{blue}{#1}}

\newcommand{\Mbar}{\overline{\ca M}}
\newcommand{\MD}{\ca M^\blacklozenge}
\newcommand{\Md}{\ca M^\lozenge}
\newcommand{\DRL}{\mathsf{DRL}}
\newcommand{\DR}{\mathsf{DR}}
\newcommand{\DRC}{\mathsf{DRC}}
\newcommand{\isom}{\stackrel{\sim}{\longrightarrow}}
\newcommand{\Ann}[1]{\on{Ann}(#1)}
\newcommand{\fm}{\mathfrak m}
\newcommand{\Mdk}{\Mbar^{\m, 1/\k}}
\newcommand{\field}{k}
\newcommand{\Mdm}{\Mbar^\m}
\newcommand{\m}{{\bd m}}
\newcommand{\cat}[1]{\bd{#1}}
\newcommand{\M}{\mathsf{M}}
\newcommand{\ghost}{\bar{\mathsf{M}}}
\newcommand{\gp}{\mathsf{gp}}
\newcommand{\fib}{\mathsf{cat}}
\newcommand{\et}{\mathsf{\acute{e}t}}
\renewcommand{\sf}[1]{\mathsf{#1}}
\newcommand{\Pic}{\mathfrak{Pic}}
\newcommand{\Sym}{\on{Sym}}
\newcommand{\Chow}{\on{CH}}
\newcommand{\Spec}{\on{Spec}}
\newcommand{\Picabs}{\mathfrak{Pic}}
\newcommand{\Picrel}{\mathfrak{Pic}^{\mathrm{rel}}}
\newcommand{\CHop}{\Chow}
\newcommand{\divCHop}{\on{divCH}}
\newcommand{\DRop}{\mathsf{DR}} 
\newcommand{\LogChow}{\on{LogCH}}
\newcommand{\divLogChow}{\on{divLogCH}}
\newcommand{\LogDR}{\sf{LogDR}}
\newcommand{\Pictdz}{\mathfrak{Jac}}
\renewcommand{\log}{\sf{log}}
\newcommand{\trop}{\sf{trop}}
\newcommand{\op}{\sf{op}}
\newcommand{\aj}{\sf{aj}}
\newcommand{\GL}{\on{GL}}
\newcommand{\J}{\sf{J}}

\section{Introduction}

If $C/S$ is a family of smooth algebraic curves and $\ca L$ on $C$ a line bundle, the double ramification cycle measures the locus of points $s \in S$ where the line bundle $\ca L$ becomes trivial upon restriction to the fibre $C_s$. More formally, $\DRop(\ca L)$ is a virtual fundamental class of this locus, living in the Chow group of codimension $g$ cycles on $S$. Extending this class in a natural way to families of (pre)stable curves, and giving a tautological formula, has been the subject of much recent research, including 
\cite{
Faber2005Relative-maps-a,
Hain2013Normal-function,
Grushevsky2012The-double-rami,
Dudin2015Compactified-un,
Farkas2016The-moduli-spac,
Schmitt2016Dimension-theor,
Janda2016Double-ramifica,
Marcus2017Logarithmic-com,
Janda2018Double-ramifica,
Holmes2017Jacobian-extens,
Holmes2017Extending-the-d,
Holmes2019Infinitesimal-s,Holmes2017Multiplicativit}. In particular, \cite{Bae2020Pixtons-formula} gives a definition of a double ramification cycle $\DRop(\ca L)$ for any line bundle $\ca L$ on any family $C/S$ of prestable curves, and proves a tautological formula for this cycle. 

\subsection{Double-double ramification cycles are tautological}
Suppose now we have two line bundles $\ca L_1$, $\ca L_2$ on a smooth curve $C/S$. Then the \emph{double-double 
 ramification cycle} $\DRop(\ca L_1, \ca L_2)$ measures the locus of $s \in S$ such that both $\ca L_1$ and $\ca L_2$ become trivial on the fibre $C_s$ -- of course, this is just the intersection of the corresponding cycles $\DRop(\ca L_1)$ and $\DRop(\ca L_2)$. The key insight of \cite{Holmes2017Multiplicativit} was that this naive intersection is the `wrong' way to extend this class to a family of (pre)stable curves. Instead, one should construct a new virtual class for the product, and in general it will not equal the product of the virtual classes of the two factors:
\begin{equation}\label{eq:DDR_neq_prod}
\DRop(\ca L_1, \ca L_2) \neq \DRop(\ca L_1) \cdot \DRop(\ca L_2). 
\end{equation}

Why is this new construction better than simply taking the intersection of the classes? One way to see this is to consider what happens when one tensors the line bundles $\ca L_1$ and $\ca L_2$ together. For a family of smooth curves one sees easily the formula
\begin{equation}\label{eq:false_mult}
\DRop(\ca L_1)\DRop(\ca L_2) = \DRop(\ca L_1)\DRop(\ca L_1\otimes \ca L_2);
\end{equation}
this also holds in compact type, which plays a key role in the construction of quadratic double ramification integrals and the noncommutative KdV hierarchy in \cite{Buryak2019Quadratic-doubl}. However, \ref{eq:false_mult} \emph{fails} for general families of (pre)stable curves, obstructing the extension of quadratic double ramification integrals beyond the compact-type case (see \cite[\S8]{Holmes2017Multiplicativit} for an explicit example of this failure). 
On the other hand, the formula 
\begin{equation}\label{eq:true_mult}
\DRop(\ca L_1,\ca L_2) = \DRop(\ca L_1,\ca L_1\otimes \ca L_2)
\end{equation}
\emph{does} hold for arbitrary families, giving hope of extending the results of \cite{Buryak2019Quadratic-doubl} beyond compact-type. This is a particular instance of a $\GL_2(\bb Z)$-invariance property for the double-double ramification cycles, which we generalise in \ref{thm:invariance_on_Mbar} to $\GL_r(\bb Z)$-invariance for $r$-fold products. 

While the cycle $\DRop(\ca L_1,\ca L_2)$ is in some ways better behaved than the product $\DRop(\ca L_1)\DRop(\ca L_2)$, until now the question of whether it is a tautological cycle has remained open, and is important to address if we hope to study quadratic double ramification integrals. Our first main theorem resolves this question: 
\begin{theorem}\label{thm:intro_DDR_taut}
Let $g$, $n$ be non-negative integers, $r$ a positive integer, and $\ca L_1, \dots, \ca L_r$ be line bundles on the universal curve over $\Mbar_{g,n}$. Then the $r$-fold double ramification cycle 
\begin{equation}
\DRop(\ca L_1, \dots, \ca L_r)
\end{equation}
lies in the tautological subring of the Chow ring of $\Mbar_{g,n}$. 
\end{theorem}
This theorem opens up the possibility of giving an explicit formula for the class $\DRop(\ca L_1, \dots, \ca L_r)$ in terms of the standard generators of the tautological ring, as was done in \cite{Janda2016Double-ramifica} for the case $r=1$ (that $\DRop(\ca L)$ lies in the tautological ring was proven earlier by Faber and Pandharipande \cite{Faber2005Relative-maps-a}, but no formula was given at that time). 

\begin{remark}
Ranganathan and Molcho have an independent approach (in a paper to appear soon) to \ref{thm:intro_DDR_taut}, by studying the virtual strict transforms of the DR cycle. 
\end{remark}

\subsection{Logarithmic Chow rings}

The fundamental reason for the failure of the product formula \ref{eq:false_mult} for stable curves is that $\DRop(\ca L)$ should not really be viewed as a cycle on $\Mbar_{g,n}$, but rather it lives naturally on a \emph{log blowup} of $\Mbar_{g,n}$ --- essentially an iterated blowup in reduced boundary strata. To avoid having to make a choice of blowup, we work on $\LogChow(\Mbar_{g,n})$, which is defined to be the colimit of all log blowups of $\Mbar_{g,n}$ and comes with a proper pushforward $\nu_*\colon \LogChow(\Mbar_{g,n}) \to \CHop(\Mbar_{g,n})$, which is a group homomorphism but \emph{not} a ring homomorphism. The construction of $\DRop(\ca L)$ can be upgraded (see \ref{def:logDR}) to give a cycle $\LogDR(\ca L) \in \LogChow(\Mbar_{g,n})$, whose pushforward to $\CHop(\Mbar_{g,n})$ is $\DRop(\ca L)$. The formula
\begin{equation}
\LogDR(\ca L_1)\LogDR(\ca L_2) = \LogDR(\ca L_1)\LogDR(\ca L_1\otimes \ca L_2)
\end{equation}
is not hard to prove in $\LogChow(\Mbar_{g,n})$ (see \ref{thm:invariance_on_Mbar}). We then define 
\begin{equation}
\DRop(\ca L_1, \dots, \ca L_r) = \nu_*\left( \LogDR(\ca L_1) \cdots \LogDR(\ca L_r)\right), 
\end{equation}
from which the product formula \ref{eq:true_mult} is immediate. The fact that \ref{eq:false_mult} fails is then just a symptom of the fact that proper pushforward $\nu_*\colon \LogChow(\Mbar_{g,n}) \to \CHop(\Mbar_{g,n})$ is not a ring homomorphism. 

\subsection{Logarithmic tautological rings}

Our proof of \ref{thm:intro_DDR_taut} (that  double-double ramification cycles are tautological) will run via showing that $\LogDR(\ca L)$ is tautological; but first we have to decide what it means for a cycle in $\LogChow(\Mbar_{g,n})$ to be tautological. 

In fact, we need to do something slightly more general. Our proof that $\LogDR(\ca L)$ is tautological for a line bundle $\ca L$ on the universal curve over $\Mbar_{g,n}$ proceeds by reduction to the fact that the usual double ramification cycle is tautological. However, for this reduction step it will be necessary to modify the universal curve (so that it is no longer stable, only prestable), and also to modify the line bundle $\ca L$. This leads us to study double ramification cycles on the total-degree-zero\footnote{In \cite{Bae2020Pixtons-formula} we do not assume total degree zero, but the DR cycle is supported in total degree zero, so this is only a superficial change. } Picard stack $\Pictdz$ of the universal curve over the stack $\frak M$ of all prestable marked curves, exactly the setting considered in \cite{Bae2020Pixtons-formula}. 

It is then necessary to define a tautological subring of $\LogChow(\Pictdz)$, which is slightly delicate as this smooth algebraic stack is neither Deligne-Mumford nor quasi-compact. For this we develop a theory of piecewise-polynomial functions on any log algebraic stack, and for log smooth stacks over a field or dedekind scheme we construct a map from the space of piecewise-polynomial functions to the log Chow ring. We then define the tautological subring as the ring generated by image of this map together with pullbacks of classes in the usual tautological ring on $\Pictdz$ (as described in \cite[definition 4]{Bae2020Pixtons-formula}). This leads to our main technical result, from which \ref{thm:intro_DDR_taut} follows easily: 
\begin{theorem}\label{thm:intro_LogDR_tautological}
$\LogDR$ lies in the tautological subring of $\LogChow(\Pictdz)$. 
\end{theorem}

In fact we prove a stronger result (\ref{thm:LogDR_tautological}); if $\ca L$ is the universal line bundle on the universal curve $\pi\colon C  \to \Pictdz$, we define the class 
\begin{equation}
\eta = \pi_*(c_1(\ca L)^2) \in \CHop(\Pictdz), 
\end{equation}
and prove that $\LogDR$ lies in the subring of $\LogChow(\Pictdz)$ generated by boundary divisors and the class $\eta$. 

%
%


\subsection{$\LogDR$ is divisorial}

Double ramification cycles in logarithmic Chow rings are also studied in the recent paper \cite{Molcho2021The-Hodge-bundl}, with a particular emphasis on the case of the trivial bundle (corresponding to the top chern class of the Hodge bundle on the moduli space of curves). The objective there is to understand the complexity of $\DRop(\ca O_C)$ in the Chow ring, in particular to understand when it can be written as a polynomial in divisor classes. It is shown that $\DRop(\ca O_C)$ cannot be written as  polynomial in divisor classes, and conjectured that $\LogDR(\ca L)$ can be written as a polynomial in divisors for all $\ca L$. As a byproduct of the proof of \ref{thm:intro_LogDR_tautological} we obtain something a little more general. The ring $\LogChow(S)$ is graded by codimension, and we write $\divLogChow(S)$ for the subring generated in degree 1. Since $\LogDR$ lies in the ring generated by $\eta$ and boundary divisors, we immediately obtain
\begin{theorem}
\begin{equation*}
\LogDR \in \divLogChow(\Pictdz). 
\end{equation*}
\end{theorem}
By pulling back to $\Mbar_{g,n}$ this proves \cite[Conjecture C]{Molcho2021The-Hodge-bundl}. 

%

\subsection{Strategy of proof}
As with many things in life, our strategy is best illustrated by carrying it out over $\Mbar_{1,2}$. We write $C$ for the universal curve with markings $p_1$, $p_2$, and we let $\ca L = \ca O(2p_1 - 2p_2)$. Then $\DRop(\ca L)$ is invariant under various changes to $\ca L$; these are listed quite exhaustively in \cite[\S 0.6]{Bae2020Pixtons-formula}. In particular, let $D$ be the prime divisor on $C$ given by the rational tails (via the isomorphism $C = \Mbar_{1,3}$ this is the closure of the locus of curves with a single rational tail and all markings on the tail). Then Invariance V of \cite[\S 0.6]{Bae2020Pixtons-formula} states that 
\begin{equation}
\DRop(\ca L) = \DRop(\ca L(D)). 
\end{equation}
Our toehold on $\LogDR$ is obtained by realising that it should satisfy a stronger invariance property, corresponding to twisting by vertical divisors which only exist after blowing up $\Mbar_{1,2}$.  Let $x \in \Mbar_{1,2}$ be the 2-marked 2-gon (\ref{fig:2_gon}), and let $\widetilde{\ca M}_{1,2}$ be the blowup of $\Mbar_{1,2}$ in $x$ (\ref{fig:2_gon_blowup}), with $\tilde C$ the pullback of $C$. The curve $C_x$ has two irreducible components $Y_1$, $Y_2$ (say $Y_1$ carries $p_1$), and the pullbacks of these to $\tilde C$ are prime divisors supported over the exceptional locus of the blowup, which we denote $\tilde Y_i$. We would like to say that $\LogDR$ satisfies the invariance 
\begin{equation}
\LogDR(\ca L) = \LogDR(\ca L(\tilde Y_1)), 
\end{equation}
but this makes no sense because $\tilde Y_1$ is only a Weil divisor, not a Cartier divisor over the `danger' points marked in \ref{fig:2_gon_blowup}. To rectify this we blow up $\tilde C$ quite carefully so that the result $\tilde{\tilde C}$ is still a prestable curve, but now $\tilde Y_1$ is a Cartier divisor on $\tilde{\tilde C}$. Then the invariance
\begin{equation}
\LogDR(\ca L) = \LogDR(\ca L(\tilde Y_1))
\end{equation}
makes sense on $\tilde{\tilde C}$, and is moreover true. 

It is at this point perhaps not clear what we have gained; we have replaced the rather simple bundle $\ca L$ on $C$ by the rather complicated $\ca L(\tilde Y_1)$ on $\tilde{\tilde C}$. The magic is that $\ca L(\tilde Y_1)$ has multidegree $\ul 0$ --- that is, it has degree zero on every irreducible component of every fibre of $\tilde{\tilde C}$ (with the exception of the danger points, which we will sweep under the carpet for now). Now, for a line bundle of multidegree $\ul 0$ the cycle $\LogDR$ is just the pullback of the corresponding $\DRop$ (\ref{lem:invariance_7_a}), and we know the latter to be tautological by Pixton's formula \cite{Bae2020Pixtons-formula}. 

\definecolor{uuuuuu}{rgb}{0.26666666666666666,0.26666666666666666,0.26666666666666666}
\definecolor{qqqqff}{rgb}{0.0,0.0,1.0}
\begin{wrapfigure}[17]{r}{0.5\textwidth} 
  \begin{center}
\begin{tikzpicture}[line cap=round,line join=round,>=triangle 45,x=0.5cm,y=0.5cm]
\clip(-8,-2) rectangle (9.89,9);
\draw (-1.2575567486914516,-0.6094175315700505)-- (3.2947599414603594,-0.5937198878109063);
\draw (-2.6082830683267932,-1.8558013606222377)-- (-0.3470934106610911,1.1644162132132425);
\draw (-0.3470934106610911,1.1644162132132425)-- (4.676152592265047,1.1801138569723866);
\draw (4.676152592265047,1.1801138569723866)-- (2.7139471223720246,-1.8966243198198738);
\draw (2.7139471223720246,-1.8966243198198738)-- (-2.6082830683267932,-1.8558013606222377);
\draw (0.0767429708358017,-1.7710431697467202)-- (2.0703437282471135,0.8033704067529263);
\draw [shift={(3.8597867559711254,5.893383742884459)}] plot[domain=2.0557173218197633:4.240773412522362,variable=\t]({1.0*3.2846227238901546*cos(\t r)+-0.0*3.2846227238901546*sin(\t r)},{0.0*3.2846227238901546*cos(\t r)+1.0*3.2846227238901546*sin(\t r)});
\draw [shift={(-1.1084407902287268,5.674656743995157)}] plot[domain=-1.3335895279112524:1.4727510495010576,variable=\t]({1.0*2.9253667060483504*cos(\t r)+-0.0*2.9253667060483504*sin(\t r)},{0.0*2.9253667060483504*cos(\t r)+1.0*2.9253667060483504*sin(\t r)});
\draw [->] (0.9538423421561176,2.799959044307189) -- (0.9822755814323633,-0.6016939718110028);
\end{tikzpicture}
  \end{center}
  \caption{Curve over $\Mbar_{1,2}$}
  \label{fig:2_gon}
\end{wrapfigure}

\definecolor{ffqqqq}{rgb}{1.0,0.0,0.0}
\definecolor{uuuuuu}{rgb}{0.26666666666666666,0.26666666666666666,0.26666666666666666}
\definecolor{qqqqff}{rgb}{0.0,0.0,1.0}
\begin{wrapfigure}[17]{r}{0.5\textwidth}
  \begin{center}
\begin{tikzpicture}[line cap=round,line join=round,>=triangle 45,x=0.5cm,y=0.5cm]
\clip(-7.219293512707992,-3.65412224507374) rectangle (12.484251285056095,7.5);
\draw (-0.5,-0.1)-- (4.052316690151814,-0.0843023562408544);
\draw (-1.850726319635342,-1.346383829052186)-- (0.41046333803036056,1.6738337447832956);
\draw (0.41046333803036056,1.6738337447832956)-- (5.433709340956501,1.6895313885424388);
\draw (5.433709340956501,1.6895313885424388)-- (3.4715038710634754,-1.3872067882498222);
\draw (3.4715038710634754,-1.3872067882498222)-- (-1.850726319635342,-1.346383829052186);
\draw (0.8342997195272535,-1.2616256381766688)-- (2.8279004769385647,1.3127879383229795);
\draw[line width = 1.3pt] (3.5250519819297237,5.12414667861977)-- (1.2059221114802503,4.200118370862558);
\draw (-1.6007263196353416,3.593616170947813)-- (0.6604633380303611,6.613833744783292);
\draw (0.6604633380303611,6.613833744783292)-- (5.683709340956501,6.629531388542435);
\draw (5.683709340956501,6.629531388542435)-- (3.7215038710634762,3.552793211750177);
\draw (3.7215038710634762,3.552793211750177)-- (-1.6007263196353416,3.593616170947813);
\draw [->] (1.91,2.95) -- (2.0,1.915);
\draw [shift={(-0.6964891103728331,2.17087973421927)}] plot[domain=0.6342953849600048:1.371519404810235,variable=\t]({1.0*3.149076961094006*cos(\t r)+-0.0*3.149076961094006*sin(\t r)},{0.0*3.149076961094006*cos(\t r)+1.0*3.149076961094006*sin(\t r)});
\draw [shift={(4.883917140396212,6.392420826521825)}] plot[domain=3.119040086939489:4.00756136045104,variable=\t]({1.0*2.4103338249347677*cos(\t r)+-0.0*2.4103338249347677*sin(\t r)},{0.0*2.4103338249347677*cos(\t r)+1.0*2.4103338249347677*sin(\t r)});
\begin{scriptsize}
\draw[color=ffqqqq] (-0.4589393493372107,2.6306291128337578) node {danger};
\draw [fill=ffqqqq] (1.5783069721991914,4.348490463805261) circle (1.5pt);
\draw [fill=ffqqqq] (2.9843583942872822,4.908714077293484) circle (1.5pt);
\draw [fill=ffqqqq] (-0.6964891103728331,2.17087973421927) circle (1.5pt);
\draw (5.7238971216853125,2.256941194255474) node {$\Mbar_{1,2}$};
\draw (5.995670153378611,7) node {$\widetilde{\ca M}_{1,2}$};
\draw[line width = 1.3pt] (-1,-3) -- (1,-3);
\draw (4,-3) node {= exceptional locus};
\end{scriptsize}
\end{tikzpicture}
  \end{center}
  \caption{$\widetilde{ \ca M}_{1,2} \to \Mbar_{1,2}$}
    \label{fig:2_gon_blowup}
\end{wrapfigure}

\subsection{Interpretation as a new invariance of $\LogDR$}

Dimitri Zvonkine asked us whether the six invariance properties listed in \cite[\S0.6]{Bae2020Pixtons-formula}, together with knowledge of $\DRop$ for families $C/S$, $\ca L$ of multidegree zero, would be enough to determine $\DRop$ completely. The answer is no, essentially because the invariances in \cite{Bae2020Pixtons-formula} do not allow us to twist by vertical divisors on $C$ coming from non-separating edges. We saw above how to rectify this in the case of $\Mbar_{1,2}$; here we give a more general statement of the new invariance satisfied by $\LogDR$. 

Let $C/S$ be a log curve with $S$ log regular, and $\ca L$ on $C$ a line bundle. We say $C/S$ is \emph{twistable}\footnote{We thank Rahul Pandharipande for suggesting this terminology} if there exists a Cartier divisor $D$ on $C$ supported over the boundary of $S$ and such that $\ca L(D)$ has multidegree $\ul 0$. We write $\LogDR(\ca L)$ for the pullback of $\LogDR$ from $\Pictdz$ along the map $S \to \Pictdz$ induced by $\ca L$, and we write $\DRop(\ca L(D))$  for the pullback of $\DRop$ from $\Pictdz$ along the map $S \to \Pictdz$ induced by $\ca L(D)$. Viewing $\DRop(\ca L(D))$ as an element of $\LogChow(S)$ by pullback, our new invariance states
\begin{equation}\label{eq:intro_invariance}
\LogDR(\ca L) = \DRop(\ca L(D)). 
\end{equation}
That this invariance holds is quite straightforward once the definitions are set up correctly, see \ref{lem:invariance_7_a}. However, there are not enough twistable families that $\LogDR$ is determined by $\DRop$ and the invariance \ref{eq:intro_invariance}; requiring multidegree $\ul 0$ over \emph{every} point in $S$ is too restrictive a condition (e.g. it fails over the `danger' points in $\widetilde{\ca M}_{1,2}$ mentioned above). Because of this we introduce in \ref{def:almost_twistable} a notion of \emph{almost twistable} families. In \ref{lem:invariance_7_b} we show the analogue of \ref{eq:intro_invariance} for almost twistable families, and in \ref{lem:extend_PL_after_alteration} we show that there are enough almost-twistable families to completely determine $\LogDR$ from $\DRop$. 



\subsection{Notation and conventions}

We work with algebraic stacks in the sense of \cite{stacks-project}, and with log structures in the sense of Fontaine-Illusie-Kato, for which we find \cite{Ogus2018Lectures-on-log} and \cite{Kato1989Logarithmic-str} particularly useful general references. The sheaf of monoids on a log scheme (or stack) $X$ is denoted $\M_X$, and the characteristic (or ghost) sheaf is denoted $\ghost_X$, with groupifications $\M_X^\gp$ and $\ghost_X^\gp$. Occasionally we write $\ul X$ for the underlying scheme (or algebraic stack) of $X$. 

We work over a field or Dedekind scheme $\field$ equipped with trivial log structure. We work in the category of fine saturated (fs) log schemes (and stacks) over $\field$. 
In \ref{thm:DDR_tautological_on_Mbar,sec:conjecture_C,sec:logDR_tautological} we assume that $\field$ has characteristic zero, so that we can apply the results of \cite{Bae2020Pixtons-formula}; it is plausible that the results would become false were this assumption omitted. 

A \emph{log algebraic stack} is an algebraic stack equipped with an (fs) log structure. 

We work almost exclusively with operational Chow groups with rational coefficients, as defined in \cite[\S2]{Bae2020Pixtons-formula}, denoted $\CHop$.

\subsection{Acknowledgements}
We are very grateful to 
Younghan Bae, 
Lawrence Barrott, 
Samouil Molcho,
Giulio Orecchia,
Rahul Pandharipande, 
Dhruv Ranganathan, 
Johannes Schmitt, 
Pim Spelier, 
and
Jonathan Wise
for numerous discussions of double ramification cycles on Picard stacks and logarithmic Chow groups. The idea of extending the multiplication formulae in \cite{Holmes2017Multiplicativit} to a $\GL_r(\bb Z)$-invariance property came up in a discussion with Adrien Sauvaget, and was further developed at the AIM workshop on Double ramification cycles and integrable systems. 

The first-named author is very grateful to Alessandro Chiodo for many extensive discussions on computing the double ramification cycle on blowups of $\Mbar_{g,n}$, which provided key motivation and examples. 


Both authors are supported by NWO grant 613.009.103. 

\section{Logarithmic Chow rings}\label{sec:Log_Chow}


\subsection{Logarithmic Chow rings of algebraic stacks}
In this section we make a slight generalisation of some of the ideas from \cite{Holmes2017Multiplicativit}, see also \cite{Molcho2021The-Hodge-bundl}. We work extensively with log schemes (and stacks) which are both regular and log regular; equivalently, with log structures that are induced by normal crossings divisors (see \cite{Nizio2006Toric-singulari}). We make quite some effort in this and other sections to avoid unnecessary separatedness or quasi-compactness assumptions, and to work with algebraic stacks in place of (for example) schemes. This is not (primarily?) due to a particular personal preference, but rather because the objects we consider (such as the stack of log curves, or its universal Picard space) make this essential. 

\begin{definition}[{\cite[Example 4.3]{Adiprasito2018Semistable-redu}}]
A morphism $f\colon X \to Y$ of log algebraic stacks is a \emph{monoidal alteration} if it is proper, log \'etale, and is an isomorphism over the locus in $Y$ where the log structure is trivial. 
\end{definition}
Examples of monoidal alterations are log blowups and root stacks. We expect that every monoidal alteration can be dominated by a composition of log blowups and root stacks, but have not written down a proof. 

\begin{definition}
Let $X$ be an algebraic stack locally of finite type over $\field$. We define $\CHop(X)$ to be the operational Chow group of $X$ with rational coefficients, using finite-type algebraic spaces as test objects, see \cite[\S2]{Bae2020Pixtons-formula} for details. 
\end{definition}

\begin{remark}
If in addition $X$ is smooth and Deligne-Mumford then the intersection pairing induces an isomorphism from the usual Chow ring of $X$ (as defined by Vistoli \cite{Vistoli1989Intersection-th}) to the operational Chow ring $\CHop(X)$. 
\end{remark}

\begin{definition}\label{def:Logch_qcpt}
Let $X$ be a log smooth stack of finite type over $\field$. We define the (operational) log Chow ring of $X$ to be 
\begin{equation}
\LogChow(X) =  \on{colim}_{\tilde X}\CHop(\tilde X), 
\end{equation}
where the colimit runs over monoidal alterations $\tilde X\to X$ with $\tilde X$ smooth over $\field$. 
\end{definition}

\begin{definition}\label{def:determination}
Let $z \in \LogChow(X)$ and let $U \hra X$ be a quasi-compact open. We say the restriction of $z$ to $U$ is \emph{determined} on a monoidal alteration $\tilde U \to U$ if there exists a cycle $z' \in \CHop(\tilde U)$ in the equivalence class of $z$ as defined in the above remark; we then call $z'$ the \emph{determination} of $z$ on $\tilde U$. 
\end{definition} 

\begin{remark}
Because we work with rational coefficients, taking the colimit over log blowups would yield the same operational Chow ring; in particular, our Log Chow ring is the same as that in \cite[\S9]{Holmes2017Multiplicativit}. Throughout the paper we use the possibility of determining a cycle on a (smooth) log blowup without further comment.
%
\end{remark}

\begin{remark}
The idea of allowing monoidal alterations rather than just log blowups was suggested to the authors by Leo Herr. It will play little role in most of the paper, but is absolutely essential in \ref{sec:logDR_tautological}, where it allows us to apply ideas of \cite{Adiprasito2018Semistable-redu} on canonical resolution of singularities. 
\end{remark}

%

\begin{remark}
The ring $\on{colim}_{\tilde X}\CHop(\tilde X)$ can be realised concretely as the disjoint union of the rings $\CHop(\tilde X)$, modulo the equivalence relation where we set cycle $z_1$ on $\tilde X_1$ and $z_2$ on $\tilde X_2$ to be equivalent if there exists a monoidal alteration $\tilde X$ dominating both $\tilde X_1$ and $\tilde X_2$ and on which the pullbacks of $z_1$ and $z_2$ coincide. 
\end{remark}

\subsection{Operations on the logarithmic Chow ring}\label{sec:operations_on_log_chow}
Throughout this subsection $X$ and $Y$ are log smooth stacks of finite type over $\field$. 

\begin{definition}[$\LogChow$ is a $\CHop$-algebra]\label{def:CH_to_LogCH_pullback}
If $\tilde X \to X$ is a log blowup then pullback gives a ring homomorphism $\CHop(X) \to \CHop(\tilde X)$. These assemble into a ring homomorphism $\CHop(X) \to \LogChow(X)$. 
\end{definition}


\begin{definition}[Pullback for $\LogChow$]\label{def:logCH_pullback}Let $f\colon X \to Y$ be a morphism and let $z\in \LogChow(Y)$. Let $\tilde Y \to Y$ be a log blowup on which $z$ is determined, with $\tilde Y$ smooth, and let $\tilde X \to X \times_Y \tilde Y$ be a log blowup which is smooth. Then the composite $\tilde X \to X$ is a log blowup, and $\tilde f \colon \tilde X \to \tilde Y$ is lci, and we have a pullback $\tilde f^!z \in \CHop(\tilde X)$. This class $\tilde f^!z$ is independent of all choices, and the construction yields a ring homomorphism \begin{equation}f^!\colon \LogChow(Y) \to \LogChow(X). \end{equation}\end{definition}

\begin{lemma}
Let $f\colon X \to Y$ be a morphism, then the following diagram commutes: 
\begin{equation}
 \begin{tikzcd}
   \CHop(Y)  \arrow[r] \arrow[d, "f^!"] & \LogChow(Y) \arrow[d, "f^!"]\\
  \CHop(X)  \arrow[r] & \LogChow(X) \\
\end{tikzcd}
\end{equation}
\end{lemma}


\begin{definition}[Pushforward from $\LogChow$ to $\CHop$]\label{def:pushforward_smooth}
Suppose $X$ is smooth, and let $z \in \LogChow(X)$. Let $\tilde X \to X$ be a log blowup on which $z$ is determined, with $\tilde X$ smooth. Then $\pi\colon \tilde X \to X$ is proper and lci, so we have a proper pushforward on operational Chow $\pi\colon \CHop(\tilde X) \to \CHop(X)$. These assemble into a pushforward 
\begin{equation}
\LogChow(X) \to \CHop(X). 
\end{equation}
\end{definition}

\subsection{Extension to the non-quasi-compact case}

\begin{definition}\label{def:log_chow}Let $X$ be a log smooth log algebraic stack over $\field$ (we no longer assume $X$ to be quasi-compact). Let $\on{qOp}(X)$ denote the category of open substacks of $X$ which are quasi-compact. We define the (operational) log Chow ring of $X$ to be 
\begin{equation}
\LogChow(X) = \on{lim}_{U \in \on{qOp}(X)} \LogChow(U). 
\end{equation}
\end{definition}

\begin{remark}
Morally, we can think of an element of $\LogChow(X)$ as an (operational) cycle on the valuativisation\footnote{See for example \cite{Kato1989Logarithmic-deg}. } of $X$, which can be everywhere-locally represented on some finite log blowup of $X$. In the absence of a good theory of Chow groups of valuativisations of algebraic stacks, we make the above definition. 
\end{remark}

\begin{remark}
All of the constructions of \ref{sec:operations_on_log_chow} carry through to this setting by restricting to suitable quasi-compact opens. We will make use of these extensions without further comment. 
\end{remark}

\section{Tautological subrings of logarithmic Chow rings}\label{sec:tautological_subrings}

In this section we develop a fairly general theory of piecewise-polynomial functions on log algebraic stacks, generalising the theory for toric varieties (for which see \cite{payne2006equivariant} and the references therein). We use these piecewise-polynomial functions to build tautological subrings of the log Chow ring. Once again we need only log blowups in this section, root stacks are unnecessary. 

In the toric case one can hope to realise every element of the Chow ring in terms of piecewise-polynomial functions, which is far from the case in the our context; for example, all piecewise-polynomials functions are zero on a scheme equipped with the trivial log structure, but the Chow ring can be large and interesting. However, in the presence of a non-trivial log structure the piecewise-polynomial functions can still generate many interesting Chow elements. 

The theory in this section was largely developed before we became aware of the related work of Molcho, Pandharipande and Schmitt \cite{Molcho2021The-Hodge-bundl}, where `normally decorated strata classes' approximately correspond to classes coming from our piecewise-polynomial functions. Their approach is probably better for writing formulae for (log) tautological classes, and ours has the advantage that piecewise-polynomials on opens can be glued (which is very useful when working on large algebraic stacks as we do in this paper; as far as we are aware the theory in \cite{Molcho2021The-Hodge-bundl} has so far only been developed for finite-type Deligne-Mumford stacks). 


\subsection{Piecewise polynomial functions}

Let $(X, \ca O_X)$ be a ringed site and $\ca M$ a sheaf of $\ca O_X$-modules. We write $\Sym \ca M$ for the sheafification of the presheaf $U \mapsto \Sym(M(U))$; it is a sheaf of $\ca O_X$-algebras. If $X$ is any site and $\ca A$ a sheaf of abelian groups, then we view $\ca A$ as a sheaf of modules for the constant sheaf of rings $\bb Z$, yielding a sheaf $\Sym \ca A$ of graded $\bb Z$-algebras. 

\begin{example}
If $X$ is a scheme and $\ca A$ is the constant sheaf $\bb Z^n$ of abelian groups, then $\Sym \ca A$ is the constant sheaf $\bb Z[x_1, \dots, x_n]$. 
\end{example}

\begin{definition}
We define the \emph{sheaf of piecewise-polynomial functions} on a log algebraic stack $S$ as
\begin{equation}
\on{PP}_S \coloneqq \Sym \ghost_S^\gp. 
\end{equation}
we write 
\begin{equation}
\on{PP}^n_S = \Sym^n \ghost_S^\gp, 
\end{equation}
for the graded pieces, and \emph{piecewise-linear} functions are
\begin{equation}\label{eq:PL_function}
\on{PP}^1_S = \Sym^1 \ghost_S^\gp = \ghost_S^\gp. 
\end{equation}
\end{definition}

\begin{remark}\leavevmode
\begin{enumerate}
\item
The sheaf $\ghost_S^\gp$ makes sense on the big strict \'etale site of $S$, so the same holds for the sheaf $\on{PP}_S$. 
\item There is natural map $\Sym^n(\ghost_S^\gp(S)) \to \on{PP}_S^n(S)$, but is in general not surjective unless $n=1$, see \ref{eg:nodal_cubic}; this will play a prominent role in what follows. 
\item Given a map of log algebraic stacks $f\colon S' \to S$ there is a natural map $f^*\ghost_S \to \ghost_{S'}$, inducing a natural map of sheaves of $\bb Z$-algebras $f^*\on{PP}_S \to \on{PP}_{S'}$.
\end{enumerate}
\end{remark}

\begin{example}\label{eg:nodal_cubic}
Let $\ul S = \bb P_\field^2$, and let $E$ be an irreducible nodal cubic in $\ul S$, with complement $i\colon U \hra \ul S$. We define $\M_S = i_*\ca O_U$, so that $\ghost_S(S) = \bb N$, and $\Sym(\ghost_S^\gp(S))  = \bb Z[e]$, where $e$ corresponds to the divisor $E$. There is an \'etale chart for $S$ at the singular point of $E$ given by $\field[\Span{a,b}]$ where $a$, $b$ correspond to the two branches of $E$ through the singular point. The image of $\Sym^2(\ghost_S^\gp(S))$ is the free module $\bb Z\Span{(a+b)^2}$. However, there is another global section of $\on{PP}^2_S$ given by $ab$, and in fact $\on{PP}^2_S(S) = \bb Z\Span{(a+b)^2, ab} = \bb Z\Span{a^2+b^2, ab}$. 
\end{example}

\subsection{Simple log algebraic stacks}
\subsubsection{Barycentric subdivision}
If $S$ is a regular log regular log algebraic stack then by \cite[5.2]{Nizio2006Toric-singulari} there exists a unique normal crossings divisor $Z$ on $S$ (the \emph{boundary divisor} of $S$) with complement $i\colon U \to S$ and $\M_S = i_*\ca O_U$. We write the irreducible components of $Z$ as $(D_i)_{i \in I}$. 

If $S$ is a regular log regular atomic\footnote{In the sense of \cite{Abramovich2018Birational-inva}: $S$ has a unique stratum that is closed and connected, and the restriction of the characteristic monoid to this stratum is a constant sheaf.} log scheme then we define the \emph{barycentric ideal sheaf} to be the product 
\begin{equation*}
\prod_{J \sub I} \ca I(\bigcap_{j \in J} D_j), 
\end{equation*}
and the \emph{barycentric subdivision of $S$} to be the blowup of $S$ along the barycentric ideal sheaf. This blowup is stable under strict smooth pullback, defining a barycentric subdivision of any log regular log algebraic stack. A more explicit description can be found in \cite[\S 5.3]{Molcho2021The-Hodge-bundl}

\subsubsection{Simple log algebraic stacks}
\begin{definition}
If $S$ is a regular log regular log algebraic stack with boundary divisor $Z = \bigcup_{i \in I} D_i$, we say $S$ is \emph{simple} if for every $J \sub I$ the fibre product 
\begin{equation}\label{eq:D_J}
D_J \coloneqq \bigtimes_{j \in J,S} D_j
\end{equation}
is regular and the natural map on sets of connected components $\pi_0(D_J) \to \pi_0(S)$ is injective. The closed connected substacks $D_J$ are the \emph{closed strata} of $S$. 
\end{definition}

This condition is more restrictive than requiring the boundary divisor to be a strict normal crossings divisor; consider the union of a line and a smooth conic in $\bb P^2$ meeting at two points, then the intersection is not connected. 

\subsubsection{Simplifying blowups}

\begin{lemma}\label{lem:simplifying_blowups}
Let $S$ be a log regular log algebraic stack. Then there exists a log blowup $\tilde S \to S$ such that $\tilde S$ is simple. 
\end{lemma}
The proof consists of three observations:
\begin{enumerate}
\item
By \cite{Illusie2014Gabbers-modific} there exists a \emph{functorial} resolution algorithm for log regular log schemes, hence there exists a log blowup of $S$ which is regular and log regular;
\item If $S$ is regular log regular then the barycentric subdivision has strict normal crossings boundary (i.e. the substacks $D_J$ of \ref{eq:D_J} are regular);
\item If S is regular log regular with strict normal crossings boundary then the barycentric subdivision is simple. 
\end{enumerate}

\subsubsection{Global generation on simple log schemes}
In this section and the next we prove the key technical result on piecewise-polynomial functions on log schemes and stacks. The version for stacks implies that for schemes, but the proof is a little fiddly, so for expository reasons we treat the case of schemes first (the proofs in the two cases are similar). 

If $S$ is a simple log stack with irreducible boundary divisors $D_b$, we write $D_S$ for the sheafification of the presheaf on the small Zariski site $S_{Zar}$ associating to an open $U \hra S$ the free abelian group on those $D_b$ such that $D_b \cap U \neq \emptyset$. On connected opens the values of the sheaf and the presheaf coincide. 

\begin{theorem}\label{thm:simple_global_gen_sch}
Let $S$ be a quasi-compact simple log scheme, and let $n \in \bb Z_{\ge 0}$. Then the natural map of $\bb Z$-modules 
\begin{equation}
\Sym^n(\ghost_S^\gp(S)) \to (\Sym^n\ghost_S^\gp)(S)
\end{equation}
is surjective. 
\end{theorem}
In other words, $\on{PP}_S(S)$ is just the symmetric algebra on $\ghost_S(S)$; every global piecewise-polynomial function can be written \emph{globally} as a polynomial in piecewise-linear functions. 
\begin{proof}
To simplify notation we assume $S$ connected (hence irreducible). For $n=0$ and $n=1$ the result is obvious. We write $\bb S$ for the constant sheaf on the $\bb Z$-algebra $\Sym^n(\ghost_S^\gp(S))$. Restriction gives a natural surjection
\begin{equation}
\phi\colon \bb S \to \Sym(\ghost_S^\gp)
\end{equation}
whose kernel we denote $K$, a sheaf of sub-$\bb Z$-modules of $\bb S$. Fixing an ordering on the index set $B$ of the $D_b$, we identify $\bb S$ with the constant sheaf on the free abelian group on monomials in the $D_b$. If $U \sub S$ is a connected (equivalently, non-empty) open then $K(U)$ is the free abelian group on monomials $\prod_{a \in A} D_a$ such that $\cap_{a \in A} D_a \cap U = \emptyset$. 

To prove the theorem it suffices to show that $H_\et^1(S, K) = 0$. For this we choose a finite Zariski cover $\ca U = (U_i)_I$ of $S$ by atomic log schemes (this exists because we assume $S$ simple). These $U_i$ are connected and $S$ is irreducible, so all intersections $U_{ij}$, $U_{ijk}$, ... among $U_i$ are also connected. 
The $U_i$ are evidently acyclic for $K$, so it is enough to prove vanishing of the Cech cohomology group $H^1_\ca U(S, K)$. We consider the relevant piece of the ordered Cech complex $\check{C}_{ord}^\bullet$
\begin{equation*}
\prod_{i \in I} K(U_i) \stackrel{d_1}{\longrightarrow} \prod_{i < j \in I} K(U_{ij}) \stackrel{d_2}{\longrightarrow}  \prod_{i < j < k \in I} K(U_{ijk}). 
\end{equation*}
Since $K$ is a subsheaf of a constant sheaf of free abelian groups, the kernel of $d_2$ is generated by elements $s$ where, for some triple of indices $i_0 < j_0 < k_0 \in I$, we have
\begin{itemize}
\item $s_{ij} = 0$ unless $(i,j) = (i_0,j_0)$ or $(i,j) = (j_0,k_0)$;
\item
$s_{i_0j_0}$ and $s_{j_0k_0}$ map to the same monomial in $K(U_{i_0j_0k_0})$. 
\end{itemize}
Choose such an element $s$, where to simplify the notation we assume $i_0 = 1,$ $j_0 = 2,$ $k_0 = 3$. Again remembering that $K$ is a subsheaf of the constant sheaf $\bb S$ and that $U_{123}$ is connected, we can assume that $s_{12}$ and $s_{2,3}$ are both given by the monomial $\prod_{a \in A} D_a$. Then necessarily 
\begin{equation*}
\bigcap_{a \in A} D_a \cap U_{12} = \emptyset = \bigcap_{a \in A} D_a \cap U_{23} 
\end{equation*}
and
\begin{equation}
\bigcap_{a \in A} D_a \cap U_{ij} \neq \emptyset \text{ unless } (i,j) = (1,2) \text{ or } (i,j) = (2,3). 
\end{equation}
 Further, since $\bigcap_{a \in A} D_a$ is irreducible\footnote{If it is empty there is nothing to check.} we see that at least one of 
\begin{equation}
\bigcap_{a \in A} D_a \cap U_1 \text{ and } \bigcap_{a \in A} D_a \cap U_2
\end{equation}
is empty, and at least one of 
\begin{equation}
\bigcap_{a \in A} D_a \cap U_2 \text{ and } \bigcap_{a \in A} D_a \cap U_3
\end{equation}
is empty. Hence if $\bigcap_{a \in A} D_a \cap U_2 \neq \emptyset$ then both 
\begin{equation}
\bigcap_{a \in A} D_a \cap U_1 \text{ and } \bigcap_{a \in A} D_a \cap U_3
\end{equation}
are empty, hence $\bigcap_{a \in A} D_a \cap U_{13} = \emptyset$, a contradiction. We see that the element $s' \in \prod_{i \in I} K(U_i)$ given by $s'_2 = \prod_{a \in A} D_a$ and $s'_i = 0 $ for $i \neq 2$ has $d_1(s') = s$ as required. 
\end{proof}

\subsubsection{Global generation on simple log algebraic stacks}

We now prove the analogous result for log algebraic stacks. We strongly encourage the reader to skip the proof; it is almost identical to that for schemes, except that we have to work with smooth covers. 


\begin{theorem}\label{thm:global_gen_simple_stacks}
Let $S$ be a quasi-compact simple log algebraic stack, and let $n \in \bb Z_{\ge 0}$. Then the natural map of $\bb Z$-modules 
\begin{equation}
\Sym^n(\ghost_S^\gp(S)) \to (\Sym^n\ghost_S^\gp)(S)
\end{equation}
is surjective. 
\end{theorem}
\begin{proof}
As in the proof of \ref{thm:simple_global_gen_sch} we assume $S$ connected, we write $\bb S$ for the constant sheaf on the $\bb Z$-algebra $\Sym^n(\ghost_S^\gp(S))$, and 
\begin{equation}
\phi\colon \bb S \to \Sym(\ghost_S^\gp)
\end{equation}
for the natural surjection whose kernel we denote $K$, a sheaf of sub-$\bb Z$-modules of $\bb S$. Fixing an ordering on the index set $B$ of the $D_b$, we identify $\bb S$ with the constant sheaf on the free abelian group on monomials in the $D_b$. If $f\colon U \to S$ is a smooth map from a connected scheme then $K(U)$ is the free abelian group on monomials $\prod_A D_a$ such that $\cap_A D_a \cap f(U) = \emptyset$. 

To prove the theorem it suffices to show that $H_{sm}^1(S, K) = 0$. For this we first choose a strict smooth map $f\colon S' \to S$ from a log scheme. Shrinking $S'$, we may assume that \emph{the fibre of $S'$ over the generic point of any stratum of $S$ is connected}. This condition has two crucial consequences:
\begin{enumerate}
\item $S'$ is simple; 
\item If $(U_i)_{i \in I}$ is any collection of opens of $S'$, and $A \sub B$ any subset, then
\begin{equation}
\bigcap_{a\in A} D_a \cap \bigcap_{i \in I} f(U_i) \neq \emptyset
\end{equation}
if and only if 
\begin{equation}
\bigcap_{a\in A} D_a \cap f\left(\bigcap_{i \in I} U_i\right) \neq \emptyset. 
\end{equation}
\end{enumerate}

We choose a Zariski cover of $S'$ by atomic schemes, inducing a smooth cover $\ca U = (U_i)_I$ of $S$ by atomic log schemes. This cover has the property that the fibre product of any number of the $U_i$ over $S$ is connected.

The proof now proceeds exactly as in the case when $S$ was a scheme (\ref{thm:simple_global_gen_sch}), though we write the details for completeness. The $U_i$ are acyclic for $K$, so it is enough to prove vanishing of the Cech cohomology group $H^1_\ca U(S, K)$. The relevant piece of the ordered Cech complex $\check{C}_{ord}^\bullet$ is
\begin{equation*}
\prod_{i \in I} K(U_i) \stackrel{d_1}{\longrightarrow} \prod_{i < j \in I} K(U_{ij}) \stackrel{d_2}{\longrightarrow}  \prod_{i < j < k \in I} K(U_{ijk}). 
\end{equation*}
Since $K$ is a subsheaf of a constant sheaf of free abelian groups, the kernel of $d_2$ is generated by elements $s$ where, for some triple of indices $i_0 < j_0 < k_0 \in I$, we have
\begin{itemize}
\item $s_{ij} = 0$ unless $(i,j) = (i_0,j_0)$ or $(i,j) = (j_0,k_0)$;
\item
$s_{i_0j_0}$ and $s_{j_0k_0}$ map to the same monomial in $K(U_{i_0j_0k_0})$. 
\end{itemize}
Choose such an element $s$, where to simplify the notation we assume $i_0 = 1,$ $j_0 = 2,$ $k_0 = 3$. Suppose that $s_{12}$ and $s_{2,3}$ are given by the monomial $\prod_{a \in A} D_a$, so necessarily 
\begin{equation*}
\bigcap_{a \in A} D_a \cap f(U_{12}) = \emptyset = \bigcap_{a \in A} D_a \cap f(U_{23}) 
\end{equation*}
and
\begin{equation}
\bigcap_{a \in A} D_A \cap f(U_{ij}) \neq \emptyset \text{ unless } (i,j) = (1,2) \text{ or } (i,j) = (2,3)
\end{equation}
by injectivity of the restriction maps. Then we apply property (2) above to see that at least one of
\begin{equation}
\bigcap_{a \in A} D_a \cap f(U_1) \text{ and } \bigcap_{a \in A} D_a \cap f(U_2)
\end{equation}
is empty, and at least one of 
\begin{equation}
\bigcap_{a \in A} D_a \cap f(U_2) \text{ and } \bigcap_{a \in A} D_a \cap f(U_3)
\end{equation}
is empty. Hence if $\bigcap_{a \in A} D_a \cap f(U_2) \neq \emptyset$ then both 
\begin{equation}
\bigcap_{a \in A} D_a \cap f(U_1) \text{ and } \bigcap_{a \in A} D_a \cap f(U_3)
\end{equation}
are empty, hence $\bigcap_{a \in A} D_a \cap f(U_{13}) = \emptyset$, a contradiction. We see that the element $s' \in \prod_{i \in I} K(U_i)$ given by $s'_2 = \prod_{a \in A} D_a$ and $s'_i = 0 $ for $i \neq 2$ has $d_1(s') = s$ as required. 
\end{proof}

\subsection{Map to the Chow group}
\subsubsection{Map on divisors}
For an algebraic stack $X$, we write $\on{Div}(X)$ for the monoid of isomorphism classes of pairs $(\ca L, \ell)$ where $\ca L$ is a line bundle on $X$ and $\ell \in \ca L(X)$ a section, with monoid operation given by tensor product. 

Let $S$ be a log algebraic stack and $m \in \ghost_S(S)$. The preimage $\ca O_S(-m)^\times$ of $m$ in $\M_S$ is an $\ca O_S^\times$-torsor and the log structure equips it with a map to $\ca O_S(-m)^\times \to \ca O_S$. This map admits a unique $\ca O_S^\times$-equivariant extension to a map of line bundles $\ca O_S(-m) \to \ca O_S$, where we built $\ca O_S(-m)$ from $\ca O_S(-m)^\times$ by filling in the zero section. Dualising gives a map $\ca O_S \to \ca O_S(m) \coloneqq \ca O_S(-m)^\vee$, and the image of the section $1$ of $\ca O_S$ gives a section of $\ca O_S(m)$. This defines a map 
\begin{equation*}
\ca O_S(-)\colon \ghost_S(S) \to \on{Div}(S). 
\end{equation*}
This can be upgraded to a monoidal functor of fibred symmetric monoidal categories, see \cite[\S 3.1]{Borne2012Parabolic-sheav}. Taking the (operational) first chern class yields a group homomorphism
\begin{equation}\label{eq:map_to_divisors}
\Phi^1\colon \ghost_S^\gp(S) \to \CHop^1(S), 
\end{equation}
with image contained in the subgroup generated by Cartier divisors. 


\subsubsection{The case of simple finite-type stacks}
Let $S$ be a simple log algebraic stack, smooth\footnote{If $\field$ is a field of characteristic zero then being smooth is here equivalent to being locally of finite type (since simple implies regular). } over $\field$. 
Since $S$ is regular, the intersection pairing equips the Chow group $\CHop(S)$ with a commutative ring structure. As such, the map 
\begin{equation}
\Phi^1\colon \ghost_S(S) \to \CHop^1(S)
\end{equation}
of \ref{eq:map_to_divisors} extends uniquely to a ring homomorphism 
\begin{equation}
\Phi'\colon \Sym(\ghost_S(S)) \to \CHop(S). 
\end{equation}
Since 
\begin{equation}
\Sym^n(\ghost_S^\gp(S)) \to (\Sym^n\ghost_S^\gp)(S)
\end{equation}
is surjective, and any element of the kernel evidently maps to $0$ in $\CHop(S)$, this map $\Phi'$ descends to a unique ring homomorphism
\begin{equation}
\Phi\colon (\Sym\ghost_S)(S) = \on{PP}_S(S) \to \CHop(S), 
\end{equation}
whose degree 1 part is $\Phi^1$. 

\subsubsection{The case of log smooth finite-type stacks}
Let $S$ be a quasi-compact log smooth log algebraic stack over $\field$. By \ref{lem:simplifying_blowups} there exists a log blowup $\pi\colon \tilde S \to S$ with $\tilde S$ simple. We define 
\begin{equation}
\Phi_S\colon (\Sym\ghost_S)(S) = \on{PP}_S(S) \to \CHop(S)
\end{equation}
as the composite
\begin{equation}
(\Sym\ghost_S)(S) \to \Sym\ghost_{\tilde S}(\tilde S) \stackrel{\Phi_{\tilde S}}{\longrightarrow} \CHop(\tilde S) \stackrel{\pi_*}{\longrightarrow} \CHop(S). 
\end{equation}
\Cref{lem:simplifying_blowups} actually yields a \emph{canonical} choice of log blowup $\pi$, but we should still check that the map $\Phi_S$ is independent of the choice of $\pi$ (for example, if $S$ was already simple, we don't want to have changed the map by blowing up).  
\begin{lemma}\label{lem:blowup_of_simple}
Let $\pi\colon \tilde S \to S$ be a log blowup with $S$ and $\tilde S$ simple. The diagram
\begin{equation}
 \begin{tikzcd}
  \on{PP}_{\tilde S}(\tilde S) \arrow[r, "\Phi_{\tilde S}"] & \CHop(\tilde S) \arrow[d, "\pi_*"] \\
  \on{PP}_S(S) \arrow[u, "\pi^*"] \arrow[r, "\Phi_S"] & \CHop(S)
\end{tikzcd}
\end{equation}
commutes. 
\end{lemma}
\begin{proof}
Since $S$ is simple it is enough to check this for a monomial in elements of $\ghost_S(S)$ corresponding to prime boundary divisors on $S$, say $\prod_{a \in A} D_a$. Applying $\pi^*$ corresponds to taking the total transforms of these divisors up to $\tilde S$. We then need to show that 
\begin{equation}
\pi_*(\prod_{a} \pi^*D_a) = \prod_a D_a, 
\end{equation}
which follows from the projection formula and the fact that $\pi_*\pi^*$ is the identity. 
\end{proof}

\begin{lemma}
For any log regular $S$, the map $\Phi_S$ is independent of the choice of log blowup $\pi\colon \tilde S \to S$. 
\end{lemma}
\begin{proof}
Reduce to one blowup dominating another, then apply \ref{lem:blowup_of_simple}. 
\end{proof}

\begin{example}
We resume \ref{eg:nodal_cubic}, and recall that $\on{PP}^2(S) = \bb Z\Span{(a+b)^2, ab}$. Then $(a+b)^2$ maps to $E^2 \in \CHop^2(S)$, and $ab$ maps to the class of the singular point of $E$ in $\CHop^2(S)$. 
\end{example}

\subsubsection{The case of log regular stacks locally of finite type}

Let $\ul S$ be an algebraic stack locally of finite type over $\field$, and write $\on{qOp}(\ul S)$ for the category of open substacks of $\ul S$ which are quasi-compact over $\field$, with maps over $\ul S$. 
Then one sees easily that 
\begin{equation}
\CHop(\ul S) = \on{lim}_{U \in \on{qOp}(\ul S)} \CHop(U). 
\end{equation}

\begin{lemma}\label{lem:open_immersion_compatibility}
Let $i\colon S_1 \hra S_2$ be a strict open immersion of quasi-compact log smooth log algebraic stacks over $\field$. The diagram
\begin{equation}
 \begin{tikzcd}
  \on{PP}_{S_2}(S_2) \arrow[r, "\Phi_2"]\arrow[d, "i^*"] & \CHop(S_2) \arrow[d, "i^*"]\\
    \on{PP}_{S_1}(S_1) \arrow[r, "\Phi_1"] & \CHop(S_1) 
\end{tikzcd}
\end{equation}
commutes. 
\end{lemma}
\begin{proof}
A simplifying blowup for $S_2$ pulls back to one for $S_1$, so we may assume both $S_i$ simple. Then it is enough to check the result for divisors (since both maps $i^*$ are ring homomorphisms), but this is easy. 
\end{proof}

Now let $S$ be a log smooth log algebraic stack over $\field$. 
Given $p \in \on{PP}_S(S)$ and any $U \hra S$ quasi-compact, we restrict $p$ to $p_U \in \on{PP}_{U}(U)$, yielding an element $\Phi_U(p_U) \in \CHop(U)$. By \ref{lem:open_immersion_compatibility} these glue, yielding a map 
\begin{equation}\label{eq:PP_to_CHop}
\Phi_S\colon \on{PP}_S(S) \to \CHop(S). 
\end{equation}

\subsection{Subdivided piecewise-polynomials and the log-tautological ring}

%
%
%

For a log algebraic stack $S$ we define the group of subdivided piecewise-polynomial functions as
\begin{equation}
\on{sPP}'(S) = \on{colim}_{\tilde S \to S} \on{PP}(\tilde S),
\end{equation}
where $\tilde S \to S$ runs over all log blowups of $S$. 

\begin{lemma}
The pullback $\on{PP}(S) \to \on{PP}(\tilde S)$ is injective for $\tilde S \to S$ any log blowup, so the natural maps to the colimit are injective.  
\end{lemma}
\begin{proof}It suffices to show this locally, so we reduce to the atomic case. It is then enough to check that the natural map $\ghost_S(S) \to \ghost_{\tilde S} ( \tilde S) $ is injective. This is clear from the construction of the blowup in the toric case, but any log blowup is locally a strict base-change of a toric blowup. 
\end{proof}

We define the sheaf of subdivided piecewise-polynomials $\on{sPP}$ on the small strict \'etale site of $S$ as the sheafification of the presheaf of rings $\on{sPP}'\colon U \mapsto \on{sPP}'(U)$. This sheaf property then immediately yields
\begin{equation}
\on{sPP}(S) = \on{lim}_{U \in \on{qOp}(S)}\on{sPP}(U). 
\end{equation}

%
%
%

The natural maps 
\begin{equation}
\Phi_i^{\log}\colon \on{sPP}(U_i) \to \on{colim}_{\tilde U_i}\CHop(\tilde U_i)
\end{equation}
then assemble into a ring homomorphism 
\begin{equation}\label{eq:logPhi}
\Phi^{\log}\colon \on{sPP}(S) \to \LogChow(S). 
\end{equation}

\begin{remark}
The presheaf $\on{sPP}'$ is always separated, and is a sheaf if $S$ is quasi-compact and quasi-separated; we make the above construction to avoid having to worry about finding common refinements of blowups of very large stacks. 
\end{remark}

\subsubsection{The log-tautological ring}


\begin{definition}\label{def:log_tautological}
Let $S$ be a smooth log smooth log algebraic stack over $\field$ and let $T \sub \CHop(S)$ be a subring. We define $T^\log \sub \LogChow(S)$ to be the sub-$T$-algebra of $\LogChow(S)$ generated by the image of 
\begin{equation}
\Phi^\log\colon \on{sPP}(S) \to \LogChow(S). 
\end{equation}
\end{definition}
A natural application is to take $S = \Mbar_{g, n}$ and $T$ to be the usual tautological subring of the Chow ring. We want to ensure that after carrying out our logarithmic constructions and pushing back down to $\Mbar_{g,n}$ we still have tautological classes. 

\begin{definition}\label{def:large_enough}\label{lem:pushforward_of_tectonic_log}
We say a subring $T \sub \CHop(S)$ is \emph{tectonic}\footnote{It contains many strata, which are formed by overlaps of other strata, perhaps after some things blow up...} if the pushforward of $T^\log$ from $\LogChow(S)$ to $\CHop(S)$ is equal to $T$. 
\end{definition}

%
%

Giving criteria for when a subring $T \sub \CHop(S)$ is tectonic is somewhat subtle. Certainly if a subring is tectonic then it must contain all boundary strata, and the converse holds if $S$ is simple, but not in general. Fortunately for us a precise criterion has been worked out in \cite{Molcho2021The-Hodge-bundl} for the case where $S$ is Deligne-Mumford and quasi-compact, which will be enough for our applications\footnote{It seems likely that their results (perhaps with slight modification) will also hold in the setting of smooth log smooth algebraic stacks, but we have not verified the details. }. Their criterion goes by way of defining certain \emph{normally decorated strata class} in $\CHop(S)$; the definition is somewhat lengthy, and the details will not be so important for us. We need only the following lemma, and the fact that the usual tautological ring of $\Mbar_{g,n}$ contains these normally decorated strata classes. 

\begin{lemma}\label{lem:tectonic_condition}
Suppose $S$ is Deligne-Mumford and quasi-compact. Then a subring $T \sub \CHop(S)$ is tectonic if and only if it contains all normally decorated strata class. 
\end{lemma}
\begin{proof}
Let $\nu\colon \tilde S \to S$ be a log blowup. We denote by $R^\star(S)$ the ring of normally decorated strata classes on $S$, and similarly for $\tilde S$. By definition $R^\star(\tilde S)$ contains\footnote{If we take $\tilde S$ simple then $R^\star(\tilde S)$ is in fact generated by strata, but this is not needed for our argument. } all boundary strata of $\tilde S$, and $\nu_*(R^\star(\tilde S)) \sub R^\star(S)$ by \cite[Theorem 13]{Molcho2021The-Hodge-bundl}. 
\end{proof}

\subsection{Constructing a class}\label{sec:log_fun_class}
Suppose we are given the following data: 
\begin{enumerate}
\item
a quasi-separated log smooth log algebraic stack $S/\field$ which is stratified by global quotients\footnote{In practise this last condition means we must be exclude $\frak M_{1,0}$ from our results; but this is fairly harmless since we can just consider the corresponding cycle on $\frak M_{1,1}$ with zero weighting on the new marking. };
\item $f\colon X \to S$ a birational representable log \'etale morphism;
\item $\J/\field$ an algebraic stack and $i\colon e \tra \J$ a regularly embedded closed substack (or more generally an lci morphism);
\item $\sigma\colon X \to \J$ a morphism over $\field$. 
\end{enumerate}
Suppose also that the base-change $X \times_\J e$ is proper over $S$. Then we construct a class $[\sigma^*e]_{f, \log} \in \LogChow(S)$ --- we often omit the $f$ from the notation when it is clear from context. 

The simplest case of our construction is when $X$ is smooth and $f\colon X \to S$ is a log blowup (then $X \times_\J e \to S$ is automatically proper). Then $\sigma^!e$ is a well-defined class on $X$, and automatically gives an element of $\LogChow(S)$, which we denote $[\sigma^*e]_{f, \log}$.

In the general case a little more care is needed. Because $\LogChow(S)$ is defined as a limit over quasi-compact opens, we may assume that $S$ is quasi-compact. Then by \ref{lem:compactify_to_log_blowup} there exist log blowups $\tilde S \to S$ and $\tilde X \to X$ and a strict open immersion $\tilde f\colon \tilde X \hra \tilde S$ over $f$; after further log blowup we may also assume $\tilde S$ (and hence $\tilde X$) to be smooth. 

\begin{definition}\label{def:sufficiently_fine}
We call such an $\tilde S \to S$ a \emph{sufficiently fine} log blowup (for $f\colon X \to S$), and $\tilde X \hra \tilde S$ the \emph{lift} of $X$. 
\end{definition}

Set $Z = \tilde X \times_{\sigma, \J, i} e$, and consider the diagram 
\begin{equation}
 \begin{tikzcd}
Z   \arrow[r] \arrow[d] \arrow[dd, bend right, swap, "j"]& e \arrow[d, "i"]\\
\tilde X \arrow[r, "\sigma"] \arrow[d, "\tilde f"] & \J\\
\tilde S & 
\end{tikzcd}
\end{equation}
We then define 
\begin{equation}
[\sigma^*e]_{f, \log} = j_*i^![\tilde X]. 
\end{equation}
To unravel this formula, recall that $i$ is lci so we have a gysin morphism $i^!\colon A^*(\tilde X) \to A^*Z$.  The composite $j\colon Z \to \tilde S$ is a closed immersion, in particular projective, so we have a pushforward $j_*\colon A^*Z \to A^*\tilde S$. Finally, $\tilde S$ is smooth, so the intersection product furnishes a map $A^*(S) \to \CHop(\tilde S)$, and we have a natural inclusion $\CHop(\tilde S) \hra \LogChow(S)$. 
%
%

To see that the above construction is independent of the choice of $\tilde S$ we use that gysin pullbacks along lci maps commute with each other and with projective pushforward \cite[Theorem 2.1.12 (xi)]{Kresch1999Cycle-groups-fo}

\begin{remark}\label{lem:determined_on_sufficiently_fine}
Let $\tilde S \to S$ be a sufficiently fine log blowup. Then $[\sigma^*e]_{f, \log}$ is determined on $\tilde S$ (in the sense of \ref{def:determination}). 
\end{remark}

\begin{lemma}\label{lem:cover_gives_same_DR}
Let $\phi\colon X' \to X$ be another birational representable log \'etale morphism, and write $f'\colon X' \to S$, $\sigma'\colon X' \to \J$ for the obvious composites. Then 
\begin{equation}
[\sigma^*e]_{f, \log} = [\sigma'^*e]_{f', \log} \in \LogChow(S). 
\end{equation}
\end{lemma}
\begin{proof}
Let $T \to S$ be a blowup that is sufficiently fine for $X' \to S$, and which dominates a sufficiently fine blowup for $X \to S$. Unravelling the definitions one sees that the classes agree already in $\CHop(T)$. 
\end{proof}

The key to the above construction is the following lemma, whose proof is essentially the same as that of Lemma 6.1 of \cite{Holmes2017Extending-the-d}. 

\begin{lemma}\label{lem:compactify_to_log_blowup}
Let $S$ be a regular log regular qcqs stack and $f\colon X \to S$ birational separated log \'etale representable. Then there exist log blowups $\tilde S \to S$ and $\tilde X \to X$ and a strict open immersion $\tilde X \to \tilde S$ over $S$. 
\end{lemma}
Probably this lemma is false if one drops either the quasi-compactness or quasi-separatedness assumptions, but we have not managed to write down an example, and would be interested to see one. 
\begin{proof}
Consider first the case where $S$ is an affine toric variety, given by some cone $c \isom \bb N^r$. Then $X$ is given by a fan $F$ consisting of a collection of cones contained in $c$. Let $\bar F$ be a complete fan in $c$ such that every cone in $F$ is a union of cones in $c$; after further refinement of $\bar F$ we can assume that it corresponds to a log blowup $\tilde S \to S$. The restriction of $\bar F$ to $F$ gives a log blowup $\tilde X \to X$, and a strict open immersion $\tilde X \to \tilde S$. 

In the case where $S$ is an atomic log scheme we can follow essentially the same procedure, where the cone $c\isom \bb N^r$ is replaced by the stalk of the ghost sheaf over the closed stratum of $S$. 

%
%

In the general case we can find a smooth cover $S$ by finitely many atomic patches (by quasi-compactness), and each intersection can be covered by finitely many atomic patches (by quasi-separatedness). A strict map of atomic patches just corresponds to some inclusion of a cone as a face of another cone: $\bb N^r \hra \bb N^s$. 

Given a face inclusion $\bb N^r \hra \bb N^s$ and a subdivision of $\bb N^s$, we can pull the subdivision back to a (unique) subdivision of $\bb N^r$. But also, given a subdivision of $\bb N^r$ we can turn it into a subdivision of $\bb N^s$ in a \emph{canonical} way, by taking the product. Hence if $\bb N^r \hra \bb N^s$ is a face map (where we allow $r=s$) and we have subdivisions of $\bb N^r$ and $\bb N^s$, then we can find `common refinements' to subdivisions of $\bb N^r$ and $\bb N^s$ which agree along the face map. Moreover if both starting subdivisions were log blowups then so are these common refinements. 

To conclude the proof, we just need to extend this `common refinement' procedure from a single face map to any diagram $D$ of face maps with finitely many objects. Such a diagram necessarily also has finitely many morphisms (since there are only finitely many face maps between any two cones), hence the same is true for the category $D'$ obtained by formally inverting all the maps in $D$. By the discussion in the previous paragraph we can pull back a log blowup along any map in $D'$. 

We are given a log blowup of each cone in $D$. For a fixed cone $c$ there are only finitely many pairs $d, f$ where $d$ is another cone and $f\colon d \to c$ is a morphism in $D'$. We then give $c$ the log blowup which is the superposition over all these pairs $(f, d)$ of the pullback along $f$ of the given log blowup of $d$. In this way we equip every object of $D$ with a log blowup, in such a way that these are compatible along all the face maps in $D$. These then glue to a global log blowup of $S$, which pulls back to a global log blowup of $X$. 
%
%
%
\end{proof}

\section{Logarithmic double ramification cycles}

\subsection{Notation}
Here we introduce notation needed for applying the machinery developed in the previous two sections to moduli of curves and to double ramification cycles. 
\begin{enumerate}\label{def:P} 
\item
$\frak M_{g,n}$ denotes the (smooth, algebraic) stack of prestable curves of genus $g$ with $n$ ordered disjoint smooth markings. This has a normal crossings boundary, inducing a log smooth log structure. Equivalently, this is the stack of log curves of genus $g$ and $n$ markings, with a choice of total ordering on the markings (see \cite{Kato2000Log-smooth-defo}, \cite[Appendix A]{Gross2013Logarithmic-gro}; the underlying algebraic stack is then given by the machinery of minimal log structures \cite{Gillam2012Logarithmic-sta}). 
\item $\Mbar_{g,n}$ is the open substack of $\frak M_{g,n}$ consisting of Deligne-Mumford-Knudsen stable curves. 
\item 
$\frak M = \bigsqcup_{g,n} \frak M_{g,n}$ denotes the stack of all log curves with a choice of total ordering on their markings. Often the genus and markings will not be so important to us, so we can use this more compact notation. 
\item $C$ is the universal curve over $\frak M$. We will abusively use the same notation for the tautological curve over any stack over $\frak M$ (so for example, for the universal curve over $\Mbar_{g,n}$). 
\item $\Picabs$ is the relative Picard stack of $C$ over $\frak M$; objects are pairs of a curve and a line bundle on the curve. This is smooth over $\frak M$ with relative inertia $\bb G_m$; we equip it with the strict (pullback) log structure over $\frak M$. 
\item $\Pictdz$ denotes the connected component of $\Picabs$ corresponding to line bundles of (total) degree 0 on every fibre. 
\item $\J$ is the relative coarse moduli space over $\frak M$ of the \emph{fibrewise} connected component of identity in $\Pictdz$ (or equivalently in $\Picabs)$. Over the locus of smooth curves in $\frak M$ this is an abelian variety, the classical jacobian. In general it is a semiabelian variety over $\frak M$ which parametrises isomorphism classes of line bundles on $\frak C/\frak M$ which have degree 0 on every irreducible component of every geometric fibre (sometimes we refer to this condition as having \emph{multidegree $\ul 0$}). The morphism $\J \to \frak M$ is separated, quasi-compact, and relatively representable by algebraic spaces (none of which hold for $\Picabs$ or $\Pictdz$). 
\item $\bar \J $ is the relative coarse moduli space of $\Pictdz$ over $\frak M$; it can be defined analogously to $\J$ except that we require total degree 0 instead of multidegree $\ul 0$. In particular we have an open immersion $\J \hra \bar \J$, which is an isomorphism over the locus of irreducible curves. 
\end{enumerate}





\begin{lemma}\label{lem:Picabs_quasi-sep}
${\Pictdz}$ is quasi-separated. 
\end{lemma}
\begin{proof}
First we check that $\frak M$ is quasi-separated; equivalently, that the diagonal is qcqs. In other words, if $C/S$ is a prestable curve, then $\on{Isom}_S(C)$ is qcqs over $S$ - but this is well-known. 

Now we show that ${\Pictdz}$ is quasi-separated over $\frak M$. In other words, we fix a prestable curve $C/S$ and a line bundle $\ca L$ on $C$, and look at the automorphisms of $\ca L$ over $C$; but this is just $\bb G_m$. 
\end{proof}

\subsubsection{Piecewise linear functions}

If $C/S$ is a log curve and $\alpha \in \ghost_C^\gp(C)$, the \emph{outgoing slope} at a marked section $c$ of $C/S$ is the image of $\alpha$ in the stalk of the relative characteristic monoid $\ghost_{C/S, s} = \bb N$. 

\begin{definition}
A piecewise-linear (PL) function on a log curve $C/S$ is an element $\alpha \in \ghost_C^\gp(C)$ (cf. \ref{eq:PL_function}) with all outgoing slopes vanishing\footnote{It would be cleaner to work with vertical (`unmarked') log curves, but we will make use of smooth sections of $C/S$ in other places in our arguments, so we do not wish to impose verticality. }. 
\end{definition}

The preimage of $\alpha$ in the exact sequence 
\begin{equation}
1 \to \ca O_C^\times \to \M_C^\gp \to  \ghost_C^\gp \to 1
\end{equation}
defines an associated $\bb G_m$-torsor $\ca O^\times(\alpha)$, which we compactify to a line bundle $\ca O(\alpha)$ by glueing in the $\infty$ section (this is just a choice of sign; it corresponds to $\ca O(-p)$ being an ideal sheaf, rather than its dual). 

The bundle $\ca O(\alpha)$ always has total degree zero, but rarely multidegree $\ul 0$; more precisely, it has multidegree $\ul 0$ if and only if $\ca O(\alpha)$ is a pullback from $S$, if and only if $\alpha$ is constant on geometric fibres.

\subsection{Defining $\LogDR$}
Before defining the logarithmic double ramification cycle it seems useful to summarise the construction of the `usual' double ramification cycle from \cite{Bae2020Pixtons-formula}. Various constructions of double ramification cycles are given in various places in the literature in various levels of generality (e.g. \cite{Holmes2017Extending-the-d,Marcus2017Logarithmic-com,Holmes2017Jacobian-extens,Bae2020Pixtons-formula}). They are mostly\footnote{With the exception of classes constructed by means of admissible covers, which only apply in cases when working with the sheaf of differentials.} equivalent, but descriptions of the relations between the constructions are scattered across various sketches in various papers at various levels of generality\footnote{This is in fairly large part the responsibility of the first-named author.}, making it troublesome to assemble a complete picture. Here we attempt to rectify this by giving a precise and general statement of the relation between the two most widely-used definitions, that of the first-named author by resolving rational maps, and that of Marcus and Wise via tropical divisors (in the form used in \cite{Bae2020Pixtons-formula}).

\subsubsection{Tropical divisors}\label{sec:div}

If $C/S$ is a log curve and $\alpha$ a PL function on $C$, then the line bundle $\ca O_C(\alpha)$ determines a map $S \to \Pictdz$. In this way we have an Abel-Jacobi map from the stack of pairs $(C/S, \alpha)$ to $\Pictdz$. We can see this Abel-Jacobi map as a first approximation of the double ramification cycle, but the map has relative dimension 1 (a section $\alpha$ admits no non-trivial automorphisms, whereas the line bundle $\ca O_C(\alpha)$ has a $\bb G_m$ worth of automorphisms), and hence will not induce a good Chow class on $\Pictdz$. To fix this we need a little more setup. Given a log scheme $S = (S, M_S)$, we write
\begin{equation*}
\mathbb G_m^{\trop}(S) = \Gamma(S, \ghost_S^{gp}), 
\end{equation*}
which we call the tropical multiplicative group; it can naturally be extended to a presheaf on the category $\cat{LSch}_S$ of log schemes over $S$. A \emph{tropical} line on $S$ is a $\bb G_m^{\trop}$ torsor on $S$ for the strict \'etale topology. Then a point of $\cat{Div}$ is a quadruple
\begin{equation}
(C/S, P, \alpha, \ca M)
\end{equation}
where $C/S$ is a log curve, $P$ a tropical line on $S$, $\alpha\colon C \to P$ a morphism over $S$ with zero outgoing slopes, and $\ca M$ is a line bundle on $S$. 
An isomorphism
\begin{equation}
(\pi\colon C\to S, P, \alpha, \ca M) \to (\pi\colon C\to S, P', \alpha', \ca M')
\end{equation}
in $\cat{Div}$ is an isomorphism $\pi^*\ca M(\alpha)  \to \pi^*\ca M'(\alpha)$, and the Abel-Jacobi map $\aj\colon \cat{Div} \to \Pictdz$ sends $(\pi\colon C\to S, P, \alpha, \ca M) $ to $\pi^*\ca M(\alpha)$. 

In \cite{Bae2020Pixtons-formula} we defined $\DRop \in \CHop(\Pictdz)$ to be the fundamental class of the proper log monomorphism $\cat{Div} \to \Pictdz$ (we describe this in more detail in \ref{def:DRop}. )

%
%

\subsubsection{Universal $\sigma$-extending morphisms}

Over the locus of irreducible curves in $\Pictdz$ the notions of total degree and multidegree coincide, so that $\J$ comes with a tautological map from $\Pictdz$. We can think of this as a rational map 
\begin{equation}
\sigma\colon \Pictdz \dashrightarrow \J
\end{equation}
(rational because it is only defined on the open locus of irreducible curves). 

We call a map $t \colon T \to \Pictdz$ of algebraic stacks \emph{$\sigma$-extending} if\footnote{The analogous definition in \cite{Holmes2017Extending-the-d} had the additional assumption that $T$ be normal. At the time this was needed in order to be able to apply \cite[Theorem 4.1]{Holmes2014Neron-models-an} at a certain critical step in the arguments, but since then Marcus and Wise have proven the analogue of \cite[Theorem 4.1]{Holmes2014Neron-models-an} with no regularity assumptions, see \cite[Corollary 3.6.3]{Marcus2017Logarithmic-com}. This can then be use to modify they theory of \cite{Holmes2017Extending-the-d} without a normality assumption. 
Alternatively one can reinstate the condition that $T$ be normal, and all of the subsequent discussion will go through unchanged except that we will have to insert normalisations in various places. By Costello's Theorem \cite{Herr2021Costellos-pushf} this will have no effect on the resulting cycles, but will make things much less readable, which is why we prefer to omit the condition. 
}
\begin{enumerate}
\item The pullback along $t$ of the locus of line bundles on smooth curves is schematically dense in $T$;
\item The rational map $T \dashrightarrow \J$ induced by $\sigma$ extends to a morphism (necessarily unique if exists, by separatedness of $\J$ over $\frak M$). 
\end{enumerate}

One can then show just as in \cite{Holmes2017Extending-the-d} that the category of $\sigma$-extending stacks over $\Pictdz$ has a terminal object, which we denote ${\Pictdz}^\loz$. The natural map 
\begin{equation}\label{eq:PicLoz_to_Jac}
f\colon {\Pictdz}^\loz \to \Pictdz
\end{equation}
is separated, relatively representable by algebraic spaces, of finite presentation, and an isomorphism over the locus of smooth (even treelike) curves, but it is not in general proper. The construction equips it with a map
\begin{equation}\label{map:sigma}
\sigma\colon {\Pictdz}^\loz \to \J. 
\end{equation}

\subsubsection{The functor of points of $\Pictdz^\loz$}\label{rem:exists_PL_function}
One can describe the functor of points (on the category of log schemes) of $\Pictdz^\loz$ in a manner very similar to the definition of $\cat{Div}$. Namely, a point of $\Pictdz^\loz$ is a quadruple
\begin{equation}
(C/S, P, \alpha, \ca L)
\end{equation}
where $C/S$ is a log curve, $P$ a tropical line on $S$, $\alpha\colon C \to P$ a morphism over $S$ with zero outgoing slopes, and $\ca L$ a line bundle on $C$ such that the line bundle $\ca L(\alpha)$ has \emph{multidegree $\ul 0$} on every fibre of $C/S$. 

Given such $(C/S, P, \alpha, \ca L)$, the map $S \to \J$ given by $\ca L(\alpha)$ is an extension of $\sigma$, so by the universal property of $\Pictdz^\loz$ we obtain a map from the stack of such quadruples to $\Pictdz^\loz$. To show this is an isomorphism, we may work locally (so assume $C/S$ to be nuclear in the sense of \cite{holmes2020models}, and smooth over a dense open of $S$), then $P$ can be taken trivial, and it is enough to show that the extension of $\sigma$ is given by a PL function; but this follows from \cite[Corollary 3.6.3]{Marcus2017Logarithmic-com}.

\subsubsection{Comparison}\label{subsec:comparison}
The key actor in \ref{sec:log_fun_class} is the fibre product of the diagram
\begin{equation}
 \begin{tikzcd}
  & \frak M \arrow[d, "e"]\\
  \Pictdz^\loz \arrow[r, "\sigma"] & \J. 
\end{tikzcd}
\end{equation}
From the functor-of-points description of $\Pictdz$ this fibre product is exactly given by tuples $(C/S, P, \alpha, \ca L) \in   \Pictdz^\loz$ such that $\ca L(-\alpha)$ is the pullback of some line bundle on $S$, say $\ca L(-\alpha) = \pi^*\ca M$. Giving the data of $\ca L$ or of $\ca M$ is exactly equivalent, and the tuple $(C/S, P, \alpha, \ca M)$ is exactly a point of $\cat{Div}$; in other words we have a pullback square
\begin{equation}
 \begin{tikzcd}
\cat{Div} \arrow[r] \arrow[d]  & \frak M \arrow[d, "e"]\\
  \Pictdz^\loz \arrow[r, "\sigma"] & \J. 
\end{tikzcd}
\end{equation}
Marcus and Wise show that the composite $\cat{Div} \to \Pictdz$ is proper, and in \cite{Bae2020Pixtons-formula} we define $\DRop$ to be the associated cycle on $\Pictdz$. The full construction of the operational class is a little subtle (see \cite[\S2]{Bae2020Pixtons-formula} for details), but is easy to describe for a smooth stack $S$ mapping to $\Pictdz$. 

\begin{definition}\label{def:DRop}
Let $S$ be a smooth stack and $\phi\colon S \to \Pictdz$ a morphism. Then $\phi$ is lci, so we have a gysin pullback $\phi^!\colon A^*(\cat{Div})\to A^*(\cat{Div}\times_\Pictdz S)$, and following \cite{Skowera2012Proper-pushforw} a proper pushforward $i_*\colon A^*(\cat{Div}\times_\Pictdz S) \to A^*(S)$. Since $S$ is smooth the intersection pairing furnishes a map $\cap\colon A^*(S) \to \CHop(S)$, and we define
\begin{equation}
\phi^*\DRop = \cap(i_*\phi^![\cat{Div}])c\in \CHop(S), 
\end{equation}
where $[\cat{Div}]$ denotes the fundamental class of $\cat{Div}$ as a cycle on itself. 
\end{definition}

\subsubsection{Defining $\LogDR$}\label{sec:LogDR}
We construct the cycle $\LogDR$ in $\LogChow(\Pictdz)$ as hinted at in \cite[\S3.8]{Bae2020Pixtons-formula}. We apply the construction of \ref{sec:log_fun_class}, taking $S = \Pictdz$, $X = \Pictdz^\loz$, $\J = \J$, and $\sigma = \sigma$. We need the natural map 
\begin{equation}
\Pictdz^\loz \times_\J \frak M \to \Pictdz
\end{equation}
to be proper; this can be proven in the same way as in \cite[\S5]{Holmes2017Extending-the-d}, or follows by the comparison to the construction of Marcus-Wise in \ref{subsec:comparison}. 

\begin{definition}\label{def:logDR}
The construction in \ref{sec:log_fun_class} yields a class $\LogDR \coloneqq [\sigma^*e]_{f, \log} \in \LogChow(\Pictdz)$, the \emph{log double ramification cycle}. 
\end{definition}

Comparing the constructions yields
\begin{lemma}
Applying the pushforward $\nu_* \colon \LogChow(\Pictdz) \to \CHop(\Pictdz)$ of \ref{def:pushforward_smooth} to $\LogDR \in \LogChow(\Pictdz)$ recovers the double ramification cycle $\DRop \in \CHop(\Pictdz)$ of \cite{Bae2020Pixtons-formula}. 
\end{lemma}

\subsection{Invariance of $\LogDR$ in twistable families}

Throughout this subsection, $C/S$ is a log curve over a smooth log smooth base, and $\ca L$ is a line bundle on $C$. 

\begin{definition}
We say the pair $(C/S,\ca L)$ is \emph{twistable} if there exists a PL function $\alpha$ on $C$ such that $\ca L(\alpha)$ has multidegree $\ul 0$; we call such an $\alpha$ a \emph{twisting function}. 
\end{definition}
Being twistable is equivalent to the existence of a Cartier divisor on $D$ on $C$ supported over the boundary of $S$ and such that $\ca L(D)$ has multidegree $\ul 0$; see \cite[Theorem 4.1]{Holmes2014Neron-models-an} or \cite[Corollary 3.6.3]{Marcus2017Logarithmic-com}.


\begin{lemma}\label{lem:invariance_7_a}
Let $(C/S, \ca L)$ be twistable with $\alpha$ a twisting function. Write $\phi_\ca L\colon S \to \Pictdz$ for the map induced by $\ca L$, and $\phi_{\ca L(\alpha)}\colon S \to \Pictdz$ for the map induced by $\ca L(\alpha)$. Then
\begin{equation}
\phi_{\ca L}^*\LogDR = \phi_{\ca L(\alpha)} ^* \DRop 
\end{equation}
in $\LogChow(S)$ (where we view $\phi_{\ca L(\alpha)}^* \DRop $ in $\LogChow(S)$ by pullback, cf. \ref{sec:operations_on_log_chow}). 
\end{lemma}
\begin{proof}
Write $\sigma\colon S \dashrightarrow \J$ for the rational map induced by $\ca L$. Then the identity on $S$ is the universal $\sigma$-extending morphism! More precisely, the extension is given by $\ca L(\alpha)\colon S \to \J$, and it is easily seen to be universal among extensions. We have a pullback diagram
\begin{equation}
 \begin{tikzcd}
  S \times_J e \arrow[r]\arrow[d, "j"] & e \arrow[d, "i"] \\
  S \arrow[r, "\sigma"] & \J, 
\end{tikzcd}
\end{equation}
and the definitions of $\phi_{\ca L}^*\LogDR$ and $\phi_{\ca L(\alpha)} ^* \DRop$ simplify to
\begin{equation}
\phi_{\ca L}^*\LogDR = j_*i^![S] \;\;\; \text{and}\;\;\; \phi_{\ca L(\alpha)} ^* \DRop = j_*\sigma^![e], 
\end{equation}
which are equal since $\sigma^![e] = i^![s]$ (commutativity of the intersection pairing). 
%
\end{proof}

\begin{remark}
If $(C/S,\ca L)$ is twistable then $\alpha$ is not unique, but the line bundle $\ca L(\alpha)$ is uniquely determined up to pullback from $S$. Hence $\phi_{\ca L(\alpha)} ^* \DRop $ does not depend on the choice of $\alpha$. 
\end{remark}

%


\begin{remark}
Unfortunately the notion of twistable families seems a little too restrictive; not enough of them seem to exist to determine $\LogDR$ from $\DRop$ (though we have not written down a proof). Because of this we now introduce a weaker notion. 
\end{remark}

\begin{definition}\label{def:almost_twistable}
We say $(C/S, \ca L)$ is \emph{almost twistable} if there exist a PL function $\alpha$ on $C$ and a dense open $U \sub S$ such that:
\begin{enumerate}
\item
the restriction of $\alpha$ to $U$ is a twisting function; 
\item for any trait $T$ with generic point $\eta$ and any map $T \to S$ sending $\eta$ to a point in $U$, if the map $\eta \to \J$ induced by $\ca L(\alpha)$ can be extended to a map $T \to \J$ then the map $T \to S$ factors via $U$. 
\end{enumerate}
\end{definition}

%
\begin{remark}
%
Condition (2) implies that $U$ is the largest open of $S$ such that $(C_U/U, \ca L|_{C_U}$) is twistable. However, it is not equivalent to this; the definition also captures the possibility that the family might become twistable after some blowup.

Condition (2) is equivalent to the following: suppose $u$ lies in the closure of $U$ in $\Pictdz$, and that there exists a point in the closure of the unit section of $\Pictdz$ with the same multidegree as $u$; then $u \in U$. In fact for the invariance below it would be enough to replace this with the weaker condition that the intersection of the closure of $U$ with the closure of the unit section is contained in $U$. 
\end{remark}
\begin{lemma}\label{lem:invariance_7_b}
Let $(C/S, \ca L)$ be \emph{almost} twistable with $\alpha$ a twisting function. Write $\phi_\ca L\colon S \to \Pictdz$ for the map induced by $\ca L$, and $\phi_{\ca L(\alpha)}\colon S \to \Pictdz$ for the map induced by $\ca L(\alpha)$. Then
\begin{equation}
\phi_{\ca L}^*\LogDR = \phi_{\ca L(\alpha)} ^* \DRop 
\end{equation}
in $\LogChow(S)$. 
\end{lemma}
\begin{proof}
Write $\sigma\colon S \dashrightarrow \J$ for the rational map induced by $\ca L$. Then the inclusion $U \hra S$ is the universal $\sigma$-extending morphism. More precisely, the extension is given by $\ca L(\alpha)\colon U \to \J$, and the second property of \ref{def:almost_twistable} shows it to be universal among extensions. Since $\ca L(\alpha)$ is of total degree $0$ over the whole of $S$, it defines a map $\bar\sigma\colon  S \to \bar \J$ over the whole of $S$. Consider the diagram 
\begin{equation}
 \begin{tikzcd}
  U \times_\J e \arrow[r] \arrow[d] \arrow[dd, bend right, "j", swap]& e \arrow[d] \\
  U \arrow[d, hook] \arrow[r, "\sigma"] & \J \arrow[d, hook]\\
  S \arrow[r, "\bar \sigma"] & \bar J
\end{tikzcd}
\end{equation}
where both squares are pullbacks (the top by construction, the bottom by the defining property of $U$), so that $U \times_\J e = S \times_{\bar \J} e$. 
In the notation of \ref{sec:log_fun_class} we take $S = \tilde S$ and $X = \tilde X = U$. Then the definitions of $\phi_{\ca L}^*\LogDR$ and $\phi_{\ca L(\alpha)} ^* \DRop $ simplify to 
\begin{equation}
\phi_{\ca L}^*\LogDR = j_*i^![U]\;\;\;  \text{and}\;\;\; \phi_{\ca L(\alpha)} ^* \DRop = j_*\bar\sigma^![e] = j_*\sigma^![e], 
\end{equation}
which are equal by the commutativity of the intersection pairing. 
\end{proof}

The hard work remaining in this paper is to show that there are `enough' almost-twistable families for \ref{lem:invariance_7_b} to determine $\LogDR$ from $\DRop$.

\subsection{Extending piecewise-linear functions}

Let $C/S$ be a log curve. The key to showing the existence of enough almost-twistable families will be to extend PL functions over open subsets of $S$ to PL functions over the whole of $S$, perhaps after some monoidal alteration.

\begin{lemma}\label{lem:regular_semistable_model}
Let $C/S$ be a log curve with $S$ a smooth log smooth log algebraic stack. Then there exist
\begin{enumerate}
\item
a monoidal alteration $\tilde S \to S$;
\item a subdivision $\tilde C \to C \times_S \tilde S$
\end{enumerate}
with $\tilde C/\tilde S$ a log curve and $\tilde C$ regular. 
\end{lemma}
\begin{proof}
This follows from \cite{Adiprasito2018Semistable-redu}. More precisely, their Theorem 4.4 gives a canonical monoidal resolution over schemes, which therefore applies to stacks. The argument in the proof of their Theorem 4.5 then shows that this monoidal resolution has $\tilde C$ regular. 
\end{proof}
After applying the above lemma we will show that PL functions always extend. We start by considering the case where the base $S$ is very small (\emph{nuclear} in the sense of \cite{holmes2020models}), after which we will glue to a global solution. 
\begin{lemma}\label{lem:extend_PL_nuclear}
Let $C/S$ be a regular log curve over a (regular) log regular base, with $C/S$ nuclear. Let $U \hra S$ be strict dense open and let $\alpha$ be a PL function on $C_U/U$. Then we construct an extension $\bar \alpha$ to a PL function on $C/S$, and this construction is compatible with strict open  base-change. 
\end{lemma}
\begin{proof}
Let $r$ be the rank of $\ghost_{S,s}$, and let $D_1, \dots, D_r$ be the divisorial strata of the boundary of $S$. Let $\Gamma$ be the graph of $C/S$ over the closed stratum, and let $\Gamma_i$ be the graph over the generic point of $D_i$ (obtained by contracting those edges of $\Gamma$ whose lengths differ from $D_i$). 

On each $D_i$ with non-empty intersection with $U$ we equip $\Gamma_i$ with the PL function from $\alpha$, and for the other $D_i$ we put the zero PL function. 

Now let $z$ be a stratum of $S$, with graph $\Gamma_z$, and let $N_z \sub \{1, \dots, r\}$ be such that $\overline{\{z\}} = \bigcap_{i \in N_z} D_i$. Then $\ghost_{S,z} = \bigoplus_{i \in N_z} \bb N\cdot D_i$; write $f_i\colon \bb N\cdot D_i \to \ghost_{S,z}$ for the natural inclusion. Let $v$ be a vertex of $\Gamma_z$, and for each $i \in N_z$ let $v_i$ be its image in $\Gamma_i$ under specialisation. Then we define
\begin{equation}
\bar\alpha(v) = \sum_{i \in N_z} f_i(\alpha(v_i)). 
\end{equation}
To check that $\bar \alpha$ is a PL function on $\Gamma_z$, suppose that $e$ is an edge of $\Gamma_z$ between vertices $u$ and $v$. By regularity of $C$ we know that the length of $e$ is $D_i$ for some $i \in N_z$; suppose it is $D_1$. Then $f_i(u) = f_i(v)$ for every $i \neq 1$, and $D_1 \mid \alpha(f_1(u)) - \alpha(f_1(v))$. It is easy to see that $\bar\alpha$ restricts to $\alpha$ over $U$. 

Suppose that $S'\to S$ is a strict open map such that $C_{S'}/S'$ is also nuclear and $C_{S'}$ is regular. Let $s'$ be the closed stratum of $S'$; it is enough to check the result for the restriction of $\bar \alpha$ to $\Gamma_{s'}$. Let $N' \sub \{, \dots, r\}$ be the set of those $D_i$ meeting the image of $s'$. Then each of those $D_i$ meet the image of $S'$, and their pullbacks are exactly the divisorial strata on $S'$ (so in particular the rank of $\ghost_{S', s'}$ is $\# N'$). Then $\bar \alpha$ on $\Gamma_{s'}$ is constructed by interpolating the values of $\alpha$ on the $D_i $ for $i \in N'$, regardless of whether we compute this on $S$ or on $S'$; in particular, these give the same result. 
\end{proof}

\begin{lemma}\label{lem:extend_pl}
Let $C/S$ be a log curve with $C$ (and hence $S$) regular log regular, $S$ a log algebraic stack. Let $U \hra S$ be a strict dense open immersion and $\alpha$ a PL function on $C_U$. Then there exists a PL function $\bar \alpha$ on $S$ restricting to $\alpha$. 
\end{lemma}
\begin{proof}
In \ref{lem:extend_PL_nuclear} we give a canonical choice of extension in the case where $C/S$ is nuclear, and these are compatible with smooth base-change, so descend to algebraic stacks. 
\end{proof}

\subsection{$\LogDR$ from $\DR$}\label{sec:LogDR_from_DR}

We wish to compute $\LogDR$ in $\LogChow(\Pictdz)$. 
Let $i\colon S \hra \Pictdz$ be a strict open immersion with $S$ quasi-compact, and write $C/S$ for the universal curve and $\ca L$ on $C$ for the universal line bundle. 

\begin{lemma}\label{lem:extend_PL_after_alteration}
There exist
\begin{enumerate}
\item
a monoidal alteration $\psi\colon \tilde S \to S$;
\item a subdivision $\tilde C$ of $C\times_S \tilde S$;
\end{enumerate}
such that the pair $(\tilde C/\tilde S, \psi^*\ca L)$ is almost twistable. 
\end{lemma}
\begin{proof}
Write $\sigma\colon S \dashrightarrow \J$ for the rational map induced by $\ca L$, and let $\psi_1\colon S^\loz \to S$ be the universal $\sigma$-extending morphism. 

\textbf{Claim}: there exists a representable monoidal alteration $S^{\loz\loz}$ of $S^\loz$ over which the map $\sigma \colon S^\loz \to \J$ can be represented as $\psi_1^*\ca L(\alpha)$ for some PL function $\alpha$ over $S^{\loz\loz}$. 

The claim is clear from \ref{rem:exists_PL_function} locally on $S^\loz$, but these PL functions are only unique up to addition of a PL function from the base, and so need not glue. We define $S^{\loz\loz}$ to be the subfunctor of $S^\loz$ where the maps $\alpha\colon C \to P$ (in the notation of \ref{rem:exists_PL_function}) can be chosen such that their set of values is totally ordered (in the ordering on $P$ induced by the monoid structure on $\ghost^\gp$). That this subfunctor is a representable monoidal alteration of $S^\loz$ is proven exactly as in \cite[Theorem 5.3.4]{Marcus2017Logarithmic-com} for the map $\cat{Rub} \to \cat{Div}$, to which it is closely analogous. 

Now we can construct these $\alpha$ locally on $S^{\loz\loz}$ just as before, but with $P = \bb G_M^\trop$ and the additional requirement that the smallest value taken by $\alpha$ on any vertex is zero. This makes the $\alpha$ unique, hence they glue to a global PL function, proving the claim. 

Now let $S^\bloz \to S$ be a sufficiently fine log blowup for $S^{\loz\loz} \to S$, and let $U \hra S^\bloz$ be the lift of $S^{\loz\loz}$.

Writing $C^\bloz/S^\bloz$ for the pullback of $C/S$, we apply \ref{lem:regular_semistable_model} to the $C^\bloz/S^\bloz$ to construct a monoidal alteration $\tilde S \to S^\bloz$ and a subdivision $\tilde C$ of $C^\bloz \times_{S^\bloz} \tilde S = C \times_S \tilde S$ with $\tilde C$ (and hence $\tilde S$) regular. Writing $\psi\colon \tilde S \to S$ for the composite, we claim that the pair $(\tilde C/\tilde S, \psi^*\ca L)$ is almost twistable. 

Let $\tilde U$ be the pullback of $U$ to $\tilde S$ (a twistable open), and let $\alpha$ be a twisting function over $\tilde U$. Then \ref{lem:extend_pl} implies that this $\alpha$ can be extended to a PL function over the whole of $\tilde S$. 

Now let $T \to S$ be a trait with generic point $\eta$ landing in $\tilde U$. Suppose that the map $\eta \to \J$ given by $(\psi_1^*\ca L)(\alpha)$, then it is proven in \cite[Lemma 4.3]{Holmes2017Extending-the-d} that this cannot be extended to a map $T \to \J$ unless it can already be extended to a map $T \to U$. This shows that $(\tilde C/\tilde S, \psi^*\ca L)$ is almost twistable, with $\tilde U \hra \tilde S$ the largest twistable open. 
\end{proof}

Let $(\tilde C /\tilde S$, $\psi^*\ca L)$ be as in the statement of \ref{lem:extend_PL_after_alteration}, with twisting function $\alpha$ over $\tilde S$. Then we have maps
\begin{equation}
\phi_{\ca L}\colon \tilde S \to \Pictdz \;\;\; \text{ and } \;\;\; \phi_{\ca L(\alpha)}\colon \tilde S \to \Pictdz
\end{equation}
induced by $\psi^*\ca L$ and $\psi^*\ca L (\alpha)$ respectively. 
Then
\begin{theorem}\label{thm:key_comparison}
We have an equality of cycles
\begin{equation}\label{eq:LogDR_eq_DR}
\phi_{\ca L}^*\LogDR = \phi_{\ca L(\alpha)}^*\DRop
\end{equation}
in $\LogChow(S)$. 
\end{theorem}
\begin{proof}
Immediate from \ref{lem:invariance_7_b,lem:extend_PL_after_alteration}. 
\end{proof}

\subsection{LogDR is tautological}\label{sec:logDR_tautological}

If $\ca L$ is the universal line bundle on the universal curve $\pi\colon C  \to \Pictdz$, we define the class 
\begin{equation}\label{eq:eta}
\eta = \pi_*(c_1(\ca L)^2) \in \CHop(\Pictdz). 
\end{equation}
In \cite[Definition 4]{Bae2020Pixtons-formula} we defined a tautological subring of $\CHop(\Pictdz)$; it is the $\bb Q$-span of certain decorated prestable graphs of degree $0$, as described in \cite[Section 0.3.3]{Bae2020Pixtons-formula}; in particular, it includes the class $\eta$ from \ref{eq:eta}. Here we prove that $\LogDR$ is contained in the corresponding tautological subring of $\LogChow(\Pictdz)$. In fact, we can prove something stronger\footnote{This improvement was suggested to us by Johannes Schmitt, to whom we are very grateful for permission to include it. }. We write $\bb Q[\eta] \sub \CHop(\Pictdz)$ for the sub-$\bb Q$-algebra of the Chow ring generated by the class $\eta$ of \ref{eq:eta}, and recall from \ref{def:log_tautological} that $\bb Q[\eta]^\log$ denotes the corresponding subring of $\LogChow(\Pictdz)$; we show that $\LogDR$ lies in $\bb Q[\eta]^\log$. 

Continuing in the notation of the previous subsection, we can pull back $\bb Q[\eta]^\log$ along $\phi_\ca L\colon S \to \Pictdz$ to give a subring of $\LogChow(S)$. Since $\phi_\ca L$ is strict this is equivalent to pulling back  $\bb Q[\eta] \sub \CHop(\Pictdz)$ to $\CHop(S)$, then taking the corresponding subring of $\LogChow(S)$. We denote the resulting subring $ \bb Q[\eta]_S^\log \sub \LogChow(S)$. 

\begin{lemma}\label{lem:LogDR_taut} Suppose $\field$ has characteristic zero.
The cycle $\phi_{\ca L(\alpha)}^*\DRop$ lies in $  \bb Q[\eta]_S^\log \sub \LogChow(S)$. 
\end{lemma}
We are grateful to Johannes Schmitt for pointing out an omission in an earlier version of the proof (as well as the strengthening mentioned above). 
\begin{proof}
This is an easy consequence of Pixton's formula for $\DRop$ on ${\Pictdz}^\loz$, as stated in equation (56) of \cite[\S0.7]{Bae2020Pixtons-formula}. The formula expresses $\DRop$ as a polynomial in the following classes: 
\begin{enumerate}
\item The class $\pi_*(c_1(\psi^*\ca L(\tilde \alpha))^2)$;
\item Classes $\psi_h + \psi_{h'}$  where $h$, $h'$ are the two half-edges forming an edge of a graph of $C/\tilde S$.
\end{enumerate}
It hence suffices to show that the above classes lie in $\bb Q[\eta]_S^\log$; we treat them in order:
\begin{enumerate}
\item The class $\pi_*(c_1(\psi^*\ca L(\tilde \alpha))^2)$ can be expanded as a sum 
\begin{equation*}
\pi_*(c_1(\psi^*\ca L)^2) + \pi_*(c_1(\psi^*\ca L)c_1(\ca O_C(\tilde \alpha)) + \pi_*(c_1(\ca O_C(\tilde \alpha))^2). 
\end{equation*}
The first summand is the pullback of the tautological class $\eta = \pi_*(c_1(\ca L)^2)$ from $\CHop(\Pictdz)$, hence is in $\bb Q[\eta]_S^\log$. For the second summand, we can reduce to computing $\pi_*(c_1(\psi^*\ca L)D)$ where $D$ is some vertical prime divisor on $C/\tilde S$, say with image a prime divisor $Z$ on $\tilde S$. Then for dimension reasons we see that $\pi_*(c_1(\psi^*\ca L)D)$ is an integer multiple of the class of the boundary divisor $Z$, in particular is in $\bb Q[\eta]_S^\log$. 

Finally, the class $c_1(\ca O_{C}(\tilde \alpha))$ can be written as a sum of vertical boundary divisors on $C/\tilde S$. If $D$ and $E$ are distinct vertical prime divisors then $D$ and $E$ meet properly, and their locus of intersection is a union of vertical codimension 2 loci in $C$ (which push down to zero on $\tilde S$ for dimension reasons) and horizontal boundary strata which push forward to boundary classes on $\tilde S$. 

It remains to show that $\pi_*(D^2)$ is tautological. For this, let $Z$ be the prime divisor in $\tilde S$ which is the image of $D$, and let $E$ be the vertical divisor lying over $Z$ such that $\pi^*Z = D + E$. Then $\pi_*(D^2) = \pi_*(D \cdot (\pi^*Z - E)) = \pi_*(D \cdot E)$, which reduces us to the previous case. 

\item These are exactly the first chern classes of conormal bundles to boundary divisors on $\tilde S$. As such they can be realised as self-intersections of these boundary divisors (as we assume $\tilde S$ simple), hence are in $\bb Q[\eta]_S^\log$. \qedhere
\end{enumerate}
%
\end{proof}

Putting together \ref{thm:key_comparison} and \ref{lem:LogDR_taut} we obtain 
\begin{corollary}\label{thm:LogDR_tautological} Suppose $\field$ has characteristic zero.
Write $T \sub \CHop(\Pictdz)$ for the tautological ring as in \cite[Definition 4]{Bae2020Pixtons-formula}, and $\bb Q[\eta]\sub T$ for the subring generated by the class $\eta$ from \ref{eq:eta}. Then the class $\LogDR \in \LogChow(\Pictdz)$ lies in $\bb Q[\eta]^\log \sub T^\log$. 
\end{corollary}

\subsection{Conjecture C}\label{sec:conjecture_C}

In this section we prove Conjecture C of \cite{Molcho2021The-Hodge-bundl}; we again thank Johannes Schmitt for corrections and improvements to this argument. We first set up some general notation: if $X$ is any log smooth log algebraic stack over $\field$ we write $\divCHop(X)$ for the subring of $\CHop(X)$ generated by classes of degree 1 (in other words, by divisor classes; we call it the \emph{divisorial subring}). Similarly the ring $\LogChow(X)$ is graded by codimension, and we write $\divLogChow(X)$ for the subring generated in degree 1. 

\begin{lemma}\label{lem:log_is_divisorial}
If $T \sub \divCHop(X)$ is any subring, then $T^\log \sub \divLogChow(X)$. 
\end{lemma}
\begin{proof}We may assume $X$ is quasi-compact. Let $t \in T^\log \sub \LogChow(X)$, then there exists a simple blowup $\tilde X \to X$ on which $t$ is determined; write $\tilde t \in \CHop(\tilde X)$ for the determination. By \ref{thm:global_gen_simple_stacks} there exists a polynomial $p$ in piecewise-linear functions on $\tilde X$ such that $\Phi_{\tilde X}(p) = \tilde t$, where $\Phi_{\tilde X}$ is the map as in \ref{eq:PP_to_CHop}. Now $\Phi_{\tilde X}$ is a ring homomorphism and piecewise-linear functions map to classes of degree 1, so the result follows. 
\end{proof}

\begin{theorem}[{\cite[Conjecture C]{Molcho2021The-Hodge-bundl}}] Suppose $\field$ is a field of characteristic zero. Then
$\LogDR$ lies in $\divLogChow$. 
\end{theorem}
\begin{proof}
We know $\LogDR\in \bb Q[\eta]^\log$, and the class $\eta$ (defined in \ref{eq:eta}) has degree 1, so the result follows from \ref{lem:log_is_divisorial}. 
\end{proof}

\section{The double-double ramification cycle}

\subsection{Iterated double ramification cycles}
Let $r$ be a positive integer, and let $\Pictdz^r$ be the fibre product of $r$ copies of $\Pictdz$ over $\frak M$. This is smooth and log smooth, and comes with $r$ projection maps to $\Pictdz$. According to \ref{def:logCH_pullback} we can pull back $\LogDR$ along each of the projection maps, yielding $r$ elements of $\LogChow(\Pictdz^r)$. We define $\LogDR_r$ to be the product of these elements in the ring $\LogChow(\Pictdz^r)$. 

We can also give a more direct construction of $\LogDR_r$. Write $\J^r$ for the $r$-fold fibre product of $\J$ with itself over $\frak M$, with $e_r$ the unit section. Then over the locus of smooth curves we have a tautological morphism $\sigma_r\colon \Pictdz^r \to \J^r$, we view it as a rational map $\sigma_r\colon \Pictdz^r \dashrightarrow \J^r$, and let ${\Pictdz}^{r \loz}$ be the universal $\sigma_r$-extending stack over $\Pictdz^r$. The pullback $\sigma_r^*e_r$ is proper over $\frak M$, so we can apply the construction in \ref{sec:log_fun_class} to obtain a class $[\sigma_r^*e_r]_{\log} \in \LogChow(\Pictdz^r)$. 


\begin{lemma}\label{lem:twodefLogDRr}
These two constructions of $\LogDR_r$ coincide, i.e. 
\begin{equation}
\LogDR_r = [\sigma_r^*e_r]_{\log}. 
\end{equation}
\end{lemma}
\begin{proof}
We begin by comparing ${\Pictdz}^{r \loz}$ with $({\Pictdz}^\loz)^r$, where the latter denotes the $r$-fold fibre product over $\frak M$ in the category of fs log algebraic stacks. The composites ${\Pictdz}^{r \loz} \to \J^r \to \J$ are $\sigma$-extending, hence the universal property furnishes $r$ maps ${\Pictdz}^{r \loz} \to {\Pictdz}^\loz$, hence a map 
\begin{equation}\label{eq:compare_r_powers}
{\Pictdz}^{r \loz} \to ({\Pictdz}^\loz)^r
\end{equation}
to the fibre product. On the other hand, the fibre product $({\Pictdz}^\loz)^r$ is $\sigma_r$-extending, yielding an inverse to \ref{eq:compare_r_powers}. 

The claimed equality of cycles is then immediate from the construction in \ref{sec:log_fun_class} and an application of \cite[example 6.5.2]{Fulton1984Intersection-th} (whose proof carries over to this setting essentially unchanged). 
%
\end{proof}

Write $\ca L_1, \dots, \ca L_r$ for the tautological line bundles on the universal curve over $\Pictdz^r$, with corresponding classes 
\begin{equation}
\eta_i = \pi_*(c_1(\ca L)^2) \in \CHop(\Pictdz), 
\end{equation}
and let $\bb Q[\eta^r]$ denote the sub-$\bb Q$-algebra of $\CHop(\Pictdz^r)$ generated by these classes. From \ref{thm:LogDR_tautological} and the first construction of $\LogDR_r$ we obtain
\begin{lemma}\label{lem:r_LogDR_tautological}
\begin{equation}
\LogDR_r \in \bb Q[\eta^r]^\log \sub \LogChow(\Pictdz^r). 
\end{equation}
\end{lemma}

\subsection{$\GL_r(\bb Z)$-invariance}\label{sec:invariance}

Let $G/S$ be a commutative group scheme and $M$ an $r \times r$ matrix with integer coefficients. Writing $G^{\times_S^r}$ for the fibre product of $G$ with itself $r$ times over $S$, we write 
\begin{equation}
[M]\colon G^{\times_S^r}\to G^{\times_S^r}
\end{equation}
for the endomorphism induced by $M$. If $M \in \GL_r(\bb Z)$ then this is an automorphism. 

Applying this to $\Pictdz$ over $\frak M$ with $M \in \GL_r(\bb Z)$ yields an automorphism 
\begin{equation}
[M] \colon \Pictdz^r \to \Pictdz^r, 
\end{equation}
and pulling back along the map yields an automorphism 
\begin{equation}\label{eq:M_pullback_LPic}
[M]^*\colon \LogChow(\Pictdz^r) \to \LogChow(\Pictdz^r). 
\end{equation}

\begin{theorem}\label{thm:GL_invariance_on_Pic}
The map $[M]^*$ of \ref{eq:M_pullback_LPic} takes $\LogDR_r$ to itself. 
\end{theorem}
\begin{proof}
For this we use the second construction of $\LogDR_r$, going via ${\Pictdz}^{r \loz}$. We write $\sigma_r\colon \Pictdz^r \dashrightarrow \J^r$, and we write $e$ for the unit section of $\J^r$. We define $\Pictdz^{M\loz}$ to be the limit (in the fs category) of the solid diagram
\begin{equation}
 \begin{tikzcd}
&  \Pictdz^{r\loz} \arrow[r, "\nu"] & \Pictdz^r \arrow[dd, "{[M]}"]\\
\Pictdz^{M\loz} \arrow[ur, dotted, "s"]\arrow[dr, dotted, "t"] & &\\
&    \Pictdz^{r \loz} \arrow[r, "\nu"] & \Pictdz^r
\end{tikzcd}
\end{equation}
(we can think of $\Pictdz^{M\loz}$ as the common refinement of $\Pictdz^{r\loz}$ with its translation along $[M]$). 
Now the composite $\nu \circ s$ is $\sigma_r$-extending, as is the composite $\nu \circ t$, so we obtain a commutative diagram
\begin{equation}
 \begin{tikzcd}
& \Pictdz^r \arrow[dd, "{[M]}"] \arrow[r, dashed] & \J^r \arrow[dd, "{[M]}"]\\
\Pictdz^{M\loz} \arrow[ur, "\nu\circ s"]\arrow[dr, "\nu \circ t"]  \arrow[urr, "\sigma_s", bend left = 60] \arrow[drr, "\sigma_t", bend right = 60, swap]& \\
 & \Pictdz^r  \arrow[r, dashed] & \J^r.  
\end{tikzcd}
\end{equation}
Now $\sigma_s^*e$ is a cycle on $\Pictdz^{M\loz}$, which can induce (following \ref{sec:log_fun_class}) a logarithmic cycle on $\Pictdz^r$ in two ways; either via the map $\nu \circ s$ or via the map $\nu \circ t$. Our notation is $[\sigma_s^*e]_{\nu \circ s, \log}$ for the former and $[\sigma_s^*e]_{\nu \circ t, \log}$ for the latter, and we define analogously $[\sigma_t^*e]_{\nu \circ s, \log}$ and $[\sigma_t^*e]_{\nu \circ t, \log}$, all elements of $\LogChow(\Pictdz_r)$. Applying lemma \ref{lem:twodefLogDRr} and commutativity of the diagram yields the relations
\begin{equation}
\LogDR_r = [\sigma_s^*e]_{\nu \circ s, \log} = [\sigma_t^*e]_{\nu \circ t, \log}, 
\end{equation}
\begin{equation}
[M^{-1}]^*\LogDR_r = [\sigma_s^*e]_{\nu \circ t, \log} \;\;\; \text{and} \;\;\; [M]^*\LogDR_r = [\sigma_t^*e]_{\nu \circ s, \log}. 
\end{equation}
Finally, we note that $[M]^*e = e$ and $\sigma_t = [M] \circ \sigma_s$, so that
\begin{equation}
[M]^*\LogDR_r = [\sigma_t^*e]_{\nu \circ s, \log} = [\sigma_s^*M^*e]_{\nu \circ s, \log} = [\sigma_s^*e]_{\nu \circ s, \log} = \LogDR_r. \qedhere
\end{equation}
\end{proof}

\begin{remark}One can alternatively prove this theorem by appealing to the invariance of $\cat{Div}$ (see \ref{sec:div}) under the action of $M$. 
\end{remark}

%

\subsection{On the moduli space of curves}
Here we translate the above results into the setting of \cite{Holmes2017Multiplicativit}. We fix non-negative integers $g$, $n$ and a positive integer $r$. We choose $r$ line bundles $\ca L_1, \dots, \ca L_r$ of total degree zero on the universal curve $C$ over $\Mbar_{g,n}$, of the form 
\begin{equation}\label{eq:franchetta}
\ca L_i = \omega^{k_i}(-\sum_{j=1}^n a_{i,j} x_j)
\end{equation}
where $a_{i,1}, \dots, a_{i,n}$ are integers summing to $k_i(2g-2)$. The tuple $\ca L_1, \dots, \ca L_r$ defines a morphism 
\begin{equation}
\Psi\colon \Mbar_{g,n} \to \Pictdz^r, 
\end{equation}
and we denote
\begin{equation}
\LogDR(\ca L_1, \dots, \ca L_r) = \Psi^*\LogDR \in \LogChow(\Mbar_{g,n}), 
\end{equation}
and 
\begin{equation}
\DRop(\ca L_1, \dots, \ca L_r) = \nu_* \LogDR(\ca L_1, \dots, \ca L_r) \in \CHop(\Mbar_{g,n}). 
\end{equation}

\begin{theorem}[DDR is tautological]\label{thm:DDR_tautological_on_Mbar}Suppose $\field$ has characteristic zero. The cycle $\DRop(\ca L_1, \dots, \ca L_r)$ lies in the tautological subring of $\CHop(\Mbar_{g,n})$. 
\end{theorem}
\begin{proof}
We define the tautological ring of $\Pictdz^r$ to be the ring generated by the pullbacks of the tautological rings of the factors $\Pictdz$. 
Then the pullback along $\Psi$ of the tautological ring of $\Pictdz^r$ coincides with the tautological ring $T$ of $\Mbar_{g,n}$, since each $\ca L_i$ is of the form \ref{eq:franchetta}. \Cref{lem:r_LogDR_tautological} then implies that $\LogDR(\ca L_1, \dots, \ca L_r)$ lies in $T^\log$. Now $T$ is tectonic by \ref{lem:tectonic_condition}, and so $\DRop(\ca L_1, \dots, \ca L_r) \in T$ by \ref{lem:pushforward_of_tectonic_log}. 
\end{proof}
Note that the ring $\bb Q[\eta]\sub \CHop(\Mbar_{g,n})$ is not in general tectonic, so we really need to work with the tautological ring here. 
\begin{theorem}[$\GL(\bb Z)$-invariance of DDR]\label{thm:invariance_on_Mbar}
If $M \in \GL_r(\bb Z)$ and 
\begin{equation*}
M [\ca L_1, \dots, \ca L_r] = [\ca F_1, \dots, \ca F_r], 
\end{equation*}
then 
\begin{equation}
\DRop(\ca L_1, \dots, \ca L_r) = \DRop(\ca F_1, \dots, \ca F_r). 
\end{equation}
\end{theorem}
\begin{proof}
Immediate from \ref{thm:GL_invariance_on_Pic}. 
\end{proof}

In the case $r=2$ this recovers \cite[Theorem 1.2]{Holmes2017Multiplicativit}. 

\bibliographystyle{alpha} 
\bibliography{prebib.bib}

\vspace{+16 pt}

\noindent David~Holmes, Rosa~Schwarz\\
\textsc{Mathematisch Instituut, Universiteit Leiden, Postbus 9512, 2300 RA Leiden, Netherlands} \\
  \textit{E-mail addresses}: \texttt{holmesdst@math.leidenuniv.nl}, \texttt{r.m.schwarz@math.leidenuniv.nl}

\end{document}